\documentclass[12pt]{amsart}
\usepackage{amsfonts}
\usepackage{amsmath}
\usepackage{amssymb}
\usepackage{mathtools}
\usepackage[utf8]{inputenc}
\usepackage[english]{babel}
\usepackage{enumerate}
\usepackage[a4paper, total={7in, 8in}]{geometry}
\usepackage{mathrsfs}
\usepackage{tikz-cd}

\theoremstyle{plain}
\newtheorem{theorem}{Theorem}[section]
\newtheorem{corollary}[theorem]{Corollary}
\newtheorem{lemma}[theorem]{Lemma}
\newtheorem{sublemma}[theorem]{Sublemma}
\newtheorem{proposition}[theorem]{Proposition}
\newtheorem{property}[theorem]{Property}
\theoremstyle{definition}
\newtheorem{definition}[theorem]{Definition}
\newtheorem{remark}[theorem]{Remark}

\newtheorem{notation}[theorem]{Notation}

\numberwithin{equation}{section}
\input{tcilatex}

\begin{document}
\title[Theta functions and adiabatic curvature]{Theta functions and adiabatic curvature\\
on an Abelian variety }
\author{Ching-Hao Chang}
\address{Department of Mathematics and Applied Mathematics, Xiamen
University Malaysia, Selangor\\
\hspace*{8pt} Malaysia}
\email{chinghao.chang@xmu.edu.my}
\author{Jih-Hsin Cheng}
\address{Institute of Mathematics, Academia Sinica and National Center for
Theoretical Sciences, \hspace*{8pt} Taipei, Taiwan}
\email{cheng@math.sinica.edu.tw}
\author{I-Hsun Tsai}
\address{Department of Mathematics, National Taiwan University, Taipei,
Taiwan}
\email{ihtsai@math.ntu.edu.tw}
\keywords{theta function, Abelian variety, Picard variety, connection,
curvture, Fourier-Mukai\\
\hspace*{118pt} transform, Hermitian-Einstein metric.}

\begin{abstract}
For an ample line bundle $L$ on an Abelian variety $M$, we study the theta
functions associated with the family of line bundles $L\otimes T$ on $M$
indexed by $T\in \text{Pic}^{0}(M)$. Combined with an appropriate
differential geometric setting, this leads to an explicit curvature
computation of the direct image bundle $E$ on $\text{Pic}^{0}(M)$, whose
fiber $E_{T}$ is the vector space spanned by the theta functions for the
line bundle $L\otimes T$ on $M$. Some algebro-geometric properties of $E$
are also remarked.
\end{abstract}

\maketitle



\section{\textbf{Introduction}}

\noindent \hspace*{12pt} For an Abelian variety $M$ we write $\widehat{M}$ for the
Picard variety $\text{Pic}^{0}(M)$ with natural projections $\pi
_{1}:M\times \widehat{M}\rightarrow M$ and $\pi _{2}:M\times \widehat{M}%
\rightarrow \widehat{M}$, and $P$ for the Poincar\'{e} line bundle on $%
M\times \widehat{M}$. Let $L$ be an ample line bundle on $M$. By considering 
$\pi _{1}^{\ast }L\otimes P$ on $M\times \widehat{M}$ one defines a vector
bundle $E$ to be $E:={\pi _{2}}_{\ast }(\pi _{1}^{\ast }L\otimes P)$ on $%
\widehat{M}$ (regarded as a Fourier-Makai transform of $L$, cf. \cite{P2}).
We also study the pull-back $E^{\prime }:=\varphi _{L}^{\ast }E$ where $%
\varphi _{L}:M\rightarrow \widehat{M}$ is the standard isogeny induced by
the line bundle $L$, namely, $\varphi _{L}(\mu )$ is the line bundle $\tau
_{\mu }^{\ast }L\otimes L^{\ast }\in \text{Pic}^{0}(M)$ via the translation $%
\tau _{\mu }$ induced by $\mu \in M$. Our first result concerns an
algebro-geometric property of $E^{\prime }$: \newline

\noindent

\begin{theorem}
(See (\ref{E'KL0}) )(cf. \cite[Theorem 1.2]{CCT}). Let $E^{\prime }$ and $L$
be as above. Then $E^{\prime }\otimes L$ is a holomorphically trivial vector
bundle on $M$.
\end{theorem}

\noindent \hspace*{12pt} Our second result concerns an application of
Theorem 1.1 to the study of full curvature computation on the aforementioned 
$E$. For this purpose let us introduce a different geometric framework, cf. 
\cite{DK},\cite{CCT}. The idea is roughly described as follows. Suppose a
line bundle $G$ is given on $M$ with the first Chern class $c_{1}(G)$
represented by a translationally invariant 2-form $\omega _{G}$. It is well
known that there exists a Hermitian metric $h_{G}$ for $G$, unique up to a
constant scaling, such that $c_{1}(G;h_{G})=\omega _{G}$ as differential
forms. The similar reasoning and notation apply to $G\otimes T$ with $T\in %
\mbox{Pic}^{0}(M)$. In the case where $G$ is the above $L$, it is expected
that by suitably normalizing the metrics $h_{G\otimes T}$ with $T$ running
over $\text{Pic}^{0}(M)$, one can obtain a globally well-defined metric on
the family of line bundles $\{G\otimes T\}_{T}$ on $M$. To be more precise,
a Hermitian metric $\tilde{h}$ is defined on the bundle $\pi _{1}^{\ast
}G\otimes P$ such that it restricts to the normalized $h_{G\otimes T}$. The
point here is that $\tilde{h}$ can be explicitly described and computed.
Moreover $\tilde{h}$ naturally induces an $L^{2}$-metric $h$ on $%
H^{0}(M,G\otimes T)$ which is identified with the fiber of $E$ at $T\in 
\text{Pic}^{0}(M)$. We can now state: \newline

\noindent

\begin{theorem}
\label{thm1.2} (= Corollary \ref{ThetaE} + Theorem \ref{5.10}) (cf. \cite[%
Theorem 1.1]{CCT}) The full curvature $\Theta (E,h)$ associated with the
above metric $h$ on the bundle $E\rightarrow \widehat{M}$ is given by%
\begin{equation}
\Theta (E,h)=\Big(-\pi \overset{n}{\underset{\alpha ,\beta =1}{\sum }}%
W_{\alpha \beta }\,\frac{\delta _{\alpha }\delta _{\beta }}{\delta
_{n}\delta _{n}}\,d\widehat{\mu }_{\alpha }\wedge d\overline{\widehat{\mu }}%
_{\beta }\Big)\cdot I_{\Delta \times \Delta }\hspace*{120pt}  \label{1-1}
\end{equation}%
\noindent where $I_{\Delta \times \Delta }$ denotes the $\Delta \times
\Delta $ identity matrix. And $-c_{1}(E,h)$ is the closed, positive,
translation-invariant integral 2-form\footnote{%
A mistake for $c_{1}(E)$ in the elliptic-curve case \cite[Theorem 1.1 and
Corollary 8.5]{CCT} has been corrected in \cite{CCT2}.} $\omega ^{\vee
}=Iso^{\ast }\Big(\overset{n}{\underset{\alpha =1}{\sum }}\,\frac{\Delta }{%
\delta _{\alpha }}\,d\eta _{\alpha }\wedge d\eta _{n+\alpha }\Big)$ where $%
Iso:\widehat{M}\rightarrow M^{\ast }$ is the isomorphism defined in
Lemma\thinspace\ \ref{Iso}\vspace*{12pt} and $\Delta =\det \Delta _{\delta }=%
\overset{n}{\underset{\alpha =1}{\prod }}\delta _{\alpha }.$ See $(\ref{2.1})
$ for $\delta _{\alpha }$ and $(\ref{period})$ for $W_{\alpha \beta }$.
\end{theorem}

An immediate corollary is that $E$ admits a Hermitian-Einstein metric (i.e.
the aforementioned metric $h$) with respect to the invariant K\"{a}hler form 
$\omega ^{\vee }$ of Theorem \ref{thm1.2} on $\widehat{M}$. See \cite{Kempf2}
for related discussions aiming at computing the Hermitian-Einstein metric on
these vector bundles $E$ (alternatively called \textit{Picard bundles}
there); this work \cite{Kempf2} is based on algebraic methods involving
Mumford's theta group and invariant theory, but neither the explicit form of
the metric nor the curvature is written down or computed. With the curvature
(\ref{1-1}) it is straightforward to calculate the Chern classes and Chern
character. In this regard an analogous Picard bundle $E$ (in the terminology
of \cite{Kempf1}) on Pic$^{d}(C)$ where $C$ is a curve of genus $\geq 1$ can
be formed, and the total Chern class $C(E)$ can be explicitly determined via
the cycle $W_{d}^{r}(C)$ in Pic$^{d}(C)$ (\cite[pp. 317-319]{ACGH}) or via
the Grothendick-Riemann-Roch formula (\cite[pp. 334-336]{ACGH}). It can be
checked that Theorem \ref{thm1.2} above specialized to $\dim _{\mathbb{C}}M$ 
$=$ $1$ gives the same Chern class as that of \cite{ACGH} for $C$ an
elliptic curve. It is, however, unclear to the authors whether a purely
algebraic analogue of the explicit polarization $\omega ^{\vee }$ (cf.
Remark \ref{R-5-12}) on $\widehat{M}$ above is available in the literature.%
\newline

\noindent \hspace*{12pt} Although the idea of our proofs follows closely that of the
elliptic curve case \cite{CCT}, the technicality here is much more
laborious and a certain amount of computational details reveal their complexity
only in high dimensions. We feel it worthwhile to give a separate treatment that helps to
streamline the argument. So whenever the proof obviously duplicates the elliptic
curve case, we simply point this out and do not reproduce the proof (e.g. 
Proposition \ref{Prop5.6}); we basically give only those proofs that require
computation and reasoning not easily foreseen in the one-dimensional case.
The reader is invited to turn to \cite{CCT}; see also 
\cite{Bern}, \cite{Kempf1}, \cite{Kempf2}, \cite{P2} and \cite{To} for
related discussions. The setting of \cite{To} uses an $L^{2}$-metric in
greater generality and arrives at the \textquotedblleft projective flatness"
of the bundle $E$ (for families of polarized Abelian varieties); see also
their previous work, in particular the part on the poly-stability of Picard
bundles \cite[Theorems 3 and 4]{TW2}. For
background materials the work \cite{Pol} contains a wealth of information on
both the analytic and algebraic aspects of Abelian varieties. \newline

\noindent \hspace*{12pt} Finally let us make the following remarks. The use
of theta functions is instrumental for both proofs of Theorem 1.1 and 1.2.
By contrast, a purely algebraic proof of Theorem 1.1 over any algebraically
closed field (not necessarily of characteristic zero) can be found in \cite%
{Kempf1} by G. Kempf for different purpose. Those algebraic proofs are
basically conceptual ones. It seems natural to ask for a more down-to-earth
point of view; such a possibility is hopefully provided in this paper.
Another merit of this paper lies perhaps in our ongoing work, which is to
study certain \textquotedblleft spectral bundles" naturally arising from the
above differential geometric setting. These spectral bundles include the
above bundle $E$ as the lowest-energy level (cf. \cite[Introduction]{CCT}). 
By making use of the theta functions here, we can explore in depth the
holomorphic structure of those spectral bundles on an Abelian variety and
explicitly describe the associated eigen-sections at higher level (cf. \cite%
{CCT3}). Let us mention in passing that the theta function method here is
expected to be also relevant to analogous problems in the $p$-adic setting
(see e.g. \cite{Roq}); we hope to discuss it elsewhere in the future.\newline

\noindent \textbf{Acknowledgements}\vspace*{4pt}\newline
The first author is supported by Xiamen University Malaysia Research Fund
(Grand No. XMUMRF /2019- C3/IMAT/0010). The second author would like to
thank the Ministry of Science and Technology of Taiwan for the support:
grant no. MOST 110-2115-M-001-015 and the National Center for Theoretical
Sciences for the constant support. The third author thanks the Ministry of
Education of Taiwan for the financial support. We would also like to thank Professor
B. Berndtsson and Professor L. Weng for their interest. We are grateful to the referees
for their valuable comments and suggestions and for bringing the reference \cite{TW2}
to our attention. 

\section{\protect\bigskip \textbf{Holomorphic Line Bundles over }$%
M=V/\Lambda $}

\noindent \hspace*{12pt} In this section we basically follow the notation
and convention of \cite{GH}. Let $V$ be a complex vector space of dimension $%
n$ and $\Lambda \subseteq V$ a discrete lattice of rank $2n$. Assume that $%
M=V/\Lambda $ is an Abelian variety, and let $\omega $ be an invariant
integral form, positive of type $(1,1).$ Choose a basis $\lambda
_{1},...,\lambda _{2n}$ for $\Lambda $ over $\mathbb{Z}$ such that in terms
of dual coordinates $x_{1},...,x_{2n}$\ on $V$ 
\begin{equation}
\omega =\overset{n}{\underset{i=1}{\Sigma }}\delta _{i}\,dx_{i}\wedge
dx_{n+i}\text{, \ \ \ }\delta _{i}\in 
\mathbb{N}
\text{, \ }\delta _{i}\mid \delta _{i+1}\text{, \ }i=1,\text{ }\cdot \cdot
\cdot ,\text{ }n-1.  \label{2.1}
\end{equation}

\noindent \hspace*{12pt} With the complex basis $v_{\alpha }={\delta
_{\alpha }}^{-1}\lambda _{\alpha }$ of $V,$ $\alpha =1,...,n$, writing $%
\lambda _{\alpha }=\delta _{\alpha }v_{\alpha }$, $\lambda _{n+\alpha }=%
\overset{n}{\underset{k=1}{\Sigma }}\tau _{\alpha k}v_{k}$ we take the 
\textit{period matrix} $\Omega $ of $\Lambda \subseteq V$ to be the $n\times
2n$ matrix $(\Delta _{\delta },Z)$ where 
\begin{equation}
(\Delta _{\delta })={\scriptsize 
\begin{pmatrix}
\delta _{1} &  & 0 \\ 
& \ddots &  \\ 
0 &  & \delta _{n}%
\end{pmatrix}%
}\text{, \ }Z=%
\begin{pmatrix}
\tau _{\alpha \beta }%
\end{pmatrix}%
;\hspace*{10pt}(\func{Im}Z)^{-1}=:W=(W_{\alpha \beta }).  \label{period}
\end{equation}%
Here $Z$ is symmetric and $\func{Im}Z$ is positive definite. \newline

\begin{notation}
\label{N-2-1} A vector $v=\overset{n}{\underset{\alpha =1}{\Sigma }}%
z_{\alpha }v_{\alpha }\in V$ is also expressed by its complex coordinates $%
(z_{1},...,z_{n})$, $z_{\alpha }=z_{\alpha 1}+iz_{\alpha 2}$ on $V$, with $%
dz_{\alpha }$ dual to $v_{\alpha }$. We use the same notation $%
(z_{1},...,z_{n})$ for complex coordinates on $M$ whenever there is no
danger of confusion.\newline

\noindent One has 
\begin{equation}  \label{complaxreal}
dz_{\alpha }=\delta _{\alpha }dx_{\alpha }+\overset{n}{\underset{k=1}{\sum }%
\,}\tau _{\alpha k}\text{ }dx_{n+k}\text{, \ \ \ }d\overline{z}_{\alpha
}=\delta _{\alpha }dx_{\alpha }+\overset{n}{\underset{k=1}{\sum }\,}%
\overline{\tau }_{\alpha k}\text{ }dx_{n+k}\text{, \ \ }\alpha =1,...,n.
\end{equation}
\end{notation}

\noindent Or, equivalently (the following is needed for (\ref{5.5})) 
\begin{equation}
\begin{pmatrix}
dx_{1} \\ 
\vdots \\ 
dx_{2n}%
\end{pmatrix}%
=%
\begin{pmatrix}
\frac{i}{2}\Delta _{\delta }^{-1}\overline{Z}\,W\vspace*{8pt} \\ 
-\frac{i}{2}W%
\end{pmatrix}%
\begin{pmatrix}
dz_{1} \\ 
\vdots \\ 
dz_{2n}%
\end{pmatrix}%
+%
\begin{pmatrix}
-\frac{i}{2}\Delta _{\delta }^{-1}Z\,W\vspace*{8pt} \\ 
\frac{i}{2}W%
\end{pmatrix}%
\begin{pmatrix}
d\overline{z}_{1} \\ 
\vdots \\ 
d\overline{z}_{2n}%
\end{pmatrix}%
.\vspace*{12pt}\hspace*{60pt}  \label{realcomplex}
\end{equation}%
\vspace*{14pt} \noindent \hspace*{12pt} A holomorphic line bundle over $M$
can be described by its \textit{multipliers} $\left\{ e_{\lambda _{i}}\in 
\mathcal{O}^{\ast }(V)\right\} $ \cite[p.308]{GH}. We define $%
L_{0}\rightarrow M$ to be the holomorphic line bundle given by the
multipliers%
\begin{equation}
e_{\lambda _{\alpha }}(v)\equiv 1\text{, \ \ }e_{\lambda _{n+\alpha
}}(v)\equiv e^{-2\pi iz_{\alpha }-\pi i\tau _{\alpha \alpha }}\text{, \ \ }%
\alpha =1,...,n.\hspace*{120pt}  \label{Multipliers}
\end{equation}%
It is known that $c_{1}(L_{0})=[\omega ]$, $\omega ={\overset{n}{\underset{%
i=1}{\sum }}}\delta _{i}\,dx_{i}\wedge dx_{n+i}$ as given in $(\ref{2.1})$(%
\cite[p.310]{GH}). The associated theta functions satisfy 
\begin{equation}
\begin{cases}
\theta (z_{1,...},z_{\alpha }+\delta _{\alpha },...,z_{n})=\theta
(z_{1},...,z_{n}) \\ 
\theta (z_{1}+\tau _{\alpha 1},z_{2}+\tau _{\alpha 2},...,z_{n}+\tau
_{\alpha n})=e^{-2\pi iz_{\alpha }-\pi i\tau _{\alpha \alpha }}\theta
(z_{1},...,z_{n}),%
\end{cases}%
\alpha =1,...,n.  \label{2.3}
\end{equation}%
By the same token, a Hermitian metric $h_{L_{0}}(v)>0$ on $L_{0}$ is
charaterized by the quasi-periodic property: 
\begin{equation}
\begin{cases}
h_{L_{0}}(z_{1,...},z_{\alpha }+\delta _{\alpha
},...,z_{n})=h_{L_{0}}(z_{1},...,z_{n}) \\ 
h_{L_{0}}(z_{1}+\tau _{\alpha 1},z_{2}+\tau _{\alpha 2},...,z_{n}+\tau
_{\alpha n})=|e^{2\pi iz_{\alpha }+\pi i\tau _{\alpha \alpha
}}|^{2}h_{L_{0}}(z_{1},...,z_{n}),%
\end{cases}%
\alpha =1,...,n.  \label{2.4}
\end{equation}%
\vspace*{8pt}

\begin{lemma}
\label{Lemma 2.1} For the holomorphic line bundle $L_{0}\rightarrow M$
above, one has the quasi-periodic entire functions on $V$ 
\begin{align}
\theta _{m}(z_{1},...,z_{n}) :=\sum\limits_{k\in \mathbb{Z}^{n}}\Big( e^{\pi
i{\overset{n}{\underset{\alpha ,\beta =1}{\sum }}}k_{\alpha }k_{\beta }\tau
_{\alpha \beta }}e^{2\pi i{\overset{n}{\underset{\alpha ,\beta =1}{ \sum }}}%
\tau _{\alpha \beta }\frac{m_{\alpha }}{\delta _{\alpha }}k_{\beta }}e^{2\pi
i{\overset{n}{\underset{\alpha =1}{\sum }}}\frac{(k_{\alpha }\delta _{\alpha
}+m_{\alpha })}{\delta _{\alpha }}z_{\alpha }}\, \Big) \hspace*{100pt}
\label{2.5} \\
\mbox{where \ }k= (k_{1,}...,k_{n})\in \mathbb{Z}^{n}\text{ and\ }m\in 
\mathfrak{M}=\{(m_{1},...,m_{n})\mid \ 0\leq m_{\alpha }<\delta _{\alpha },\
m_{\alpha }\in 0\cup \mathbb{N},\ \alpha =1,...,n \},  \notag
\end{align}%
as a basis of global holomorphic sections of $L_{0}$ and $%
h_{L_{0}}(v)=e^{-2\pi \overset{n}{\underset{\ \alpha ,\beta =1}{\sum }}{W}%
_{\alpha \beta }z_{\alpha 2}z_{\beta 2}}$ as a metric on $L_{0}$ where $%
W=(W_{\alpha \beta })$ is the inverse matrix of $\func{Im}Z$.
\end{lemma}

\begin{proof}
Using the Riemann $\theta $-function $\vartheta (v)={{\underset{k\in \mathbb{%
Z}^{n}}{\sum }}}(e^{\pi i{\overset{n}{\underset{\alpha ,\beta =1}{\sum }}}%
k_{\alpha }k_{\beta }\tau _{\alpha \beta }}e^{2\pi i\overset{n}{\underset{%
\alpha =1}{\sum }}k_{\alpha }z_{\alpha }})$ (cf. \cite[p.320]{GH}) and
comparing with $(\ref{2.5})$, we have 
\begin{equation}
\theta _{m}(z_{1},...,z_{n})=e^{2\pi i{\overset{n}{\underset{\alpha =1}{\sum 
}}\frac{m_{\alpha }}{\delta _{\alpha }}z_{\alpha }}}\cdot \vartheta \Big(%
z_{1}+\overset{n}{\underset{\beta =1}{\sum }}\frac{m_{\beta }}{\delta
_{\beta }}\tau _{1\beta },...,z_{n}+\overset{n}{\underset{\beta =1}{\sum }}%
\frac{m_{\beta }}{\delta _{\beta }}\tau _{n\beta }\,\Big).  \label{2.6}
\end{equation}%
By the quasi-periodic property of $\vartheta (v)$ one finds that $\theta
_{m}(z_{1},...,z_{n})$ satisfies the quasi-periodic property (\ref{2.3})
(cf. \cite[Lemma 2.1]{CCT} for the one-dimensional case whose argument can
be easily generalized to the present case). The linear independence of $%
\left\{ \theta _{m}\right\} _{m}$ is proved later in Lemma \ref{Lemma5.2}.
With the fact $\dim H^{0}(M,L_{0})=\overset{n}{\underset{\alpha =1}{\Pi }}%
\delta _{\alpha }$ \cite[p.317]{GH} it follows that $\left\{ \theta
_{m}\right\} _{m}$ span $H^{0}(M,L_{0})$.
\end{proof}

\noindent \hspace*{12pt} For any $\mu =\overset{n}{\underset{\alpha =1}{{%
\sum }}}\,\mu _{\alpha }v_{\alpha } \in V$ where $\mu_{\alpha}=\mu_{\alpha
1} + i \mu_{\alpha 2} $, we have a map $\mathcal{\tau }_{\mu }:M\rightarrow
M $ defined by the translation by $[\mu ]\in M$. Let $L_{\mu } \coloneqq
{\mathcal{\tau }_{\mu }}^{\ast }L_{0}\rightarrow M$, which can be given by
multipliers \newline
\begin{equation}
e_{\lambda _{\alpha }}(v)\equiv 1\hspace*{6pt}\text{, }e_{\lambda _{n+\alpha
}}(v)\equiv e^{-2\pi i(z_{\alpha }+\mu _{\alpha })-\pi i\tau _{\alpha \alpha
}}\text{ \hspace*{12pt} }\alpha =1,...,n.  \label{eqnL}
\end{equation}%
\noindent Any global holomorphic sections $\tilde{\theta}$ of $L_{\mu
}\rightarrow M$ can be described via quasi-periodic entire functions $\theta 
$ on $V$ satisfying 
\begin{equation}
\begin{cases}
\theta (z_{1},...,z_{\alpha }+\delta _{\alpha },...,z_{n})=\theta
(z_{1},..,z_{n}) \\ 
\theta (z_{1}+\tau _{\alpha 1},z_{2}+\tau _{\alpha 2},...,z_{n}+\tau
_{\alpha n})=e^{-2\pi i(z_{\alpha }+\mu _{\alpha })-\pi i\tau _{\alpha
\alpha }}\ \theta (z_{1},...,z_{n}),%
\end{cases}%
\alpha =1,...,n,\vspace*{8pt}  \label{theta section}
\end{equation}%
with a metric $h_{L_{\mu }}(v)$ on $L_{\mu }\rightarrow M$: 
\begin{equation}
\begin{cases}
h_{L_{\mu }}(z_{1},...,z_{\alpha }+\delta _{\alpha },...,z_{n})=h_{L_{\mu
}}(z_{1},...,z_{n}) \\ 
h_{L_{\mu }}(z_{1}+\tau _{1\alpha },z_{2}+\tau _{2\alpha },...,z_{n}+\tau
_{n\alpha })=|e^{2\pi i(z_{\alpha }+\mu _{\alpha })+\pi i\tau _{\alpha
\alpha }}|^{2}h_{L_{\mu }}(z_{1},...,z_{n}),%
\end{cases}%
\alpha =1,...,n.\vspace*{8pt}  \label{hLmu}
\end{equation}

\noindent \hspace*{12pt} Since all the holomorphic line bundles on $M$
having the same first Chern class as that of $L_{0}$ can be represented as a
translate of $L_{0}$ \cite[p.312]{GH}, it follows from Lemma $\ref{Lemma 2.1}
$, $(\ref{theta section})$ and $(\ref{hLmu})$ that\bigskip

\begin{lemma}
\label{Lemma 2.2} Fix $\mu \in V$. For $L_{\mu }\rightarrow M$ as above, one
has the quasi-periodic entire functions on $V$ (cf. Lemma $\ref{Lemma 2.1}$
for relevant notations below)$:$ 
\begin{align*}
\theta _{m}(v;\mu )& :=\theta _{m}(v+\mu )=\sum\limits_{k\in \mathbb{Z}^{n}}%
\Big( e^{\pi i{\overset{n}{\underset{\alpha ,\beta =1}{\sum }}}k_{\alpha
}k_{\beta }\tau _{\alpha \beta }}e^{2\pi i{\overset{n}{\underset{\alpha
,\beta =1}{\sum }}}\tau _{\alpha \beta }\frac{m_{\alpha }}{\delta _{\alpha }}%
k_{\beta }}e^{2\pi i{\overset{n}{\underset{\alpha =1}{\sum }}}\frac{%
(k_{\alpha }\delta _{\alpha }+m_{\alpha })}{\delta _{\alpha }}(z_{\alpha
}+\mu _{\alpha })}\, \Big) , \\
\hspace*{10pt}\mbox{where \ }k&\in \mathbb{Z}^{n}\text{ and }m\in \mathfrak{M%
}
\end{align*}%
as a basis of $H^{0}(M,L_{\mu })$, and $h_{L_{\mu }}(v;\mu
):=h_{L_{0}}(v+\mu )=e^{-2\pi \overset{n}{\underset{\alpha ,\beta =1}{\sum }}%
{W_{\alpha \beta }}\func{Im}(z_{\alpha }+\mu _{\alpha })\func{Im}(z_{\beta
}+\mu _{\beta })}$ as a metric on $L_{\mu }$.
\end{lemma}

\bigskip \noindent \hspace*{12pt} Let $\text{Pic}^{0}(M)=\widehat{M}$ denote
the group of holomorphic line bundles on $M$ with the first Chern class
zero. In the notation of \cite[p.318]{GH} for the natural identification 
\begin{equation}  \label{Pic}
\text{Pic}^{0}(M)\cong \frac{H^{1}(M,\mathcal{O})}{H^{1}(M,\mathbb{Z})}\cong 
\frac{H^{0,1}(M)}{H^{1}(M,\mathbb{Z})}\hspace*{60pt}
\end{equation}%
the image of $H^{1}(M,\mathbb{Z})$ in $\overline{V}^{\,\ast }=H^{0,1}(M) $
is the lattice ${\overline{\Lambda }}^{\,\ast }=$ $\mathbb{Z}\{dx_{1}^{\ast
},...,dx_{2n}^{\ast }\}$ which consists of conjugate linear functionals on $%
V $ whose real part is half integral on $\Lambda \subseteq V$. Recall the
period matrix $(\Delta_{\delta}, Z)$, we write the conjugate linear parts of 
$dx_{1},...,dx_{2n}$ as ($T$ for transpose below)\vspace*{8pt} 
\begin{align}  \label{xstarzpar}
\begin{cases}
(dx_{1}^{\ast },...,dx_{n}^{\ast } )^{T} & =(\frac{-i}{2}(\Delta
_{\delta})^{-1} Z\, (\func{Im}\, Z )^{-1})(d\overline{z}_{1},...,d\overline{z%
}_{n})^{T}\hspace*{100pt}\vspace*{8pt} \\ 
(dx_{n+1}^{\ast },...,dx_{2n}^{\ast } )^{T} & =(\frac{i}{2}(\func{Im}\, Z
)^{-1})(d\overline{z}_{1},...,d\overline{z}_{n})^{T}.%
\end{cases}%
\end{align}
\newline
\noindent Reordering $\{dx_{i}^{\ast }\}$ by setting $dy_{\alpha }^{\ast
}=-dx_{n+\alpha }^{\ast }$ and $dy_{n+\alpha }^{\ast }=dx_{\alpha }^{\ast }$%
, for $\alpha =1,...,n$. If we let $v_{\alpha }^{\ast }= \frac{%
\delta_{\alpha}}{\delta_{n}} dy_{\alpha }^{\ast }$, we have $\overline{%
\Lambda }^{\,\ast }=$ $\mathbb{Z}\{dy_{1}^{\ast },...,dy_{2n}^{\ast }\}$
with $(dy_{1}^{\ast },...,dy_{2n}^{\ast })^{T}=(\delta _{n}(\Delta _{\delta
})^{-1}\mid \delta_{n}(\Delta _{\delta })^{-1}Z\, (\Delta _{\delta })^{-1})($
$v_{1}^{\ast },...,$ $v_{n}^{\ast })^{T}$. One has $\widehat{M}=\overline{V}%
^{\,\ast }/{\overline{\Lambda }}^{\,\ast }$.\newline

\noindent \hspace*{12pt} With $v_{\alpha}^{\ast}$ above, we let $(\widehat{%
\mu }_{1},...,\widehat{\mu }_{n})$ be the complex coordinates of $\widehat{%
\mu }=\overset{n}{\underset{\alpha =1}{{\sum }}}\,\widehat{\mu }_{\alpha
}v_{\alpha }^{\ast }$ on $\overline{{\normalsize V}}^{\,\ast }$ (and on $%
\widehat{M}$) $;$ $d\widehat{\mu }_{\alpha }$ are the forms dual to $%
v_{\alpha }^{\ast }$. Similarly, we write $\mu = \underset{\alpha=1}{\overset%
{n}{\sum}} \mu_{\alpha} v_{\alpha} \in V$ with $v_{\alpha}$ in the beginning
of this section. Denote by $[\widehat{\mu}]$, $[\mu]$ the class of $\widehat{%
\mu}$, $\mu$ in $\widehat{M}$, $M$ respectively.\newline

\noindent \hspace*{12pt} A well-known map $\varphi _{L_{0}}:M{\normalsize %
\rightarrow }$ Pic$^{0}(M)$ $=$ $\hat{M}$ via the translation $\tau _{\mu }:M%
{\normalsize \rightarrow M}$ with $[\mu ]\in M$ is given by $\varphi
_{L_{0}}([\mu ])=\tau _{\mu }{}^{\ast }L_{0}{\normalsize \otimes }%
L_{0}^{\ast }$ where $L_{0}^{\ast }$ is the dual line bundle of $L_{0}$.
Denote by $\widetilde{\varphi }_{L_{0}}:V{\normalsize \rightarrow }\overline{%
{\normalsize V}}^{\,\ast }$ the natural lifting map. The following property
is well known (see \cite[Property 1]{CCT} and \cite[pp. 315-317]{GH} for
more details):

\begin{property}
$\widetilde{\label{Prop 2.3}\varphi }_{L_{0}}:V{\normalsize \rightarrow }%
\overline{{\normalsize V}}^{\,\ast }$is a complex linear transformation such
that 
\begin{equation}  \label{complex linear}
\widetilde{\varphi }_{L_{0}}(v_{\alpha })=\frac{\delta_{n}}{\delta_{\alpha}}%
\, v_{\alpha }^{\ast },\hspace*{20pt} \alpha =1,...,n.\hspace*{40pt}
\end{equation}
\noindent Thus $\widehat{\mu}_{\alpha}=\frac{\delta_{n}}{\delta_{\alpha}}
\mu_{\alpha}$, $\alpha=1,...,n$, if $\widehat{\mu}=\widetilde{\varphi}%
_{L_{0}}(\mu)$.\newline
\end{property}

\noindent \hspace*{12pt} We denote by $P_{[\widehat{\mu }]}$ or simply $P_{%
\widehat{\mu }}$ the line bundle corresponding to $[\widehat{\mu }]\in 
\widehat{M}=\text{Pic}^{0}(M)$ in (\ref{Pic}). From Property \ref{Prop 2.3}
with $\widehat{\mu }=\varphi _{L_{0}}(\mu )$, one has 
\begin{equation}
P_{\widehat{\mu }}=P_{\varphi _{L_{0}}([\mu ])}\cong \tau _{\mu }{}^{\ast
}L_{0}{\normalsize \otimes }L_{0}^{\ast },\hspace*{20pt}\forall \mu \in V.
\end{equation}

\noindent Recall the Poincar\'{e} line bundle \cite[p.328]{GH}:

\begin{lemma}
\label{Poincare}There is a unique holomorphic line bundle $P{\normalsize %
\rightarrow M\times }\widehat{{\normalsize M}}$ called the Poincar\'{e} line
bundle satisfying $i)$ $P|_{M\times \{\widehat{\mu }\}}\cong P_{\widehat{\mu 
}}$ and $ii)$ $P|_{\{0\}\times \widehat{M}}$ is a holomorphically trivial
line bundle.
\end{lemma}

\section{\protect\bigskip \textbf{A holomorphic line bundle $\widetilde{K}%
\rightarrow M\times M$}}

\hspace*{12pt} From the Poincar\'{e} line bundle $P{\normalsize \rightarrow
M\times }\widehat{{\normalsize M}}$\ and the map $\varphi _{L_{0}}$ $:$ $M%
{\normalsize \rightarrow }\widehat{M}$ as introduced above, we are going to
define a holomorphic line bundle $\widetilde{K}\rightarrow M\times M$. Let $%
M_{1}=M_{2}=M$ and $\pi :V\times V\rightarrow M_{1}\times M_{2}=V/\Lambda
\times V/\Lambda $ with projections $\pi _{i}:M_{1}\times M_{2}\rightarrow
M_{i},\ i=1,2.$ We rewrite the lattice vectors $\lambda _{\alpha }$, $%
\lambda _{n+\alpha }$\ $(\alpha =1,...,n)$ of $M_{1}$ as $\lambda _{\alpha
,0}$, $\lambda _{n+\alpha ,0}$ repectively, and\ analogusly $\lambda _{\beta
}$, $\lambda _{n+\beta }$ as $\lambda _{0,\beta }$, $\lambda _{0,n+\beta }$
for the case of $M_{2}$.

\begin{definition}
Define $\widetilde{K}:=\pi _{1}^{\ast }L_{0}{\normalsize \otimes (Id\times
\varphi _{L_{0}})^{\ast }P\otimes }\pi _{2}^{\ast }L_{0}\rightarrow
M_{1}\times M_{2}$ where $P{\normalsize \rightarrow M\times }\widehat{%
{\normalsize M}}$ is the Poincar\'{e} line bundle.
\end{definition}

\begin{proposition}
\label{Prop4.2}With the notations above, a system of multipliers for $%
\widetilde{K}$ can be%
\begin{equation}
\begin{cases}
e_{\lambda _{\alpha ,0}}(v;\mu )\equiv 1,\hspace*{20pt} & e_{\lambda
_{n+\alpha ,0}}(v;\mu )=e^{-2\pi iz_{\alpha }-2\pi i\mu _{\alpha }-\pi i\tau
_{\alpha \alpha }}\hspace*{100pt} \\ 
e_{\lambda _{0,\beta }}(v;\mu )\equiv 1, & e_{\lambda _{0,n+\beta }}(v;\mu
)=e^{-2\pi iz_{\beta }-2\pi i\mu _{\beta }-\pi i\tau _{\beta \beta }}.%
\end{cases}
\label{4.1}
\end{equation}%
\vspace{4pt}
\end{proposition}

\begin{proof}
Recall that a holomorphic line bundle on $M\times M=V/\Lambda \times
V/\Lambda $ is essentially described by a system of multipliers $%
\{e_{\lambda _{\alpha ,0}},e_{\lambda _{n+\alpha ,0}},e_{\lambda _{0,\beta
},}e_{\lambda _{0,n+\beta }}\in \mathcal{O}^{\ast }(V\times V)\}$ satisfying
the compatibility relations : For $i,j=\{1,...,2n\}$ (in (\ref{4.2}) below,
we write $e_{\lambda _{i,0}}(v+\lambda _{\alpha };\mu )$ to mean $e_{\lambda
_{i,0}}(z_{1},...,z_{\alpha }+\delta _{\alpha },...,z_{n};\mu _{1},...,\mu
_{n})$, $e_{\lambda _{i,0}}(v;\mu +\lambda _{n+\beta })$ to mean $e_{\lambda
_{i,0}}(z;\mu _{1}+\tau _{1\beta },\mu _{2}+\tau _{2\beta },...,\mu
_{n}+\tau _{n\beta })$ etc. for notational convenience) 
\begin{equation}
\begin{cases}
e_{\lambda _{i,0}}(v+\lambda _{j};\mu )e_{\lambda _{j,0}}(v;\mu )=e_{\lambda
_{j,0}}(v+\lambda _{i};\mu )e_{\lambda _{i,0}}(v;\mu ) \\ 
e_{\lambda _{i,0}}(v;\mu +\lambda _{j})e_{\lambda _{0,j}}(v;\mu )=e_{\lambda
_{0,j}}(v+\lambda _{i};\mu )e_{\lambda _{i,0}}(v;\mu ) \\ 
e_{\lambda _{0,i}}(v;\mu +\lambda _{j})e_{\lambda _{0,j}}(v;\mu )=e_{\lambda
_{0,j}}(v;\mu +\lambda _{i})e_{\lambda _{0,i}}(v;\mu ).%
\end{cases}%
\hspace{90pt} \vspace*{8pt}  \label{4.2}
\end{equation}

\noindent \hspace*{12pt} Using the multipliers for ${}\pi _{1}^{\ast
}L_{0}:\{e_{\lambda _{\alpha ,0}}(v;\mu )\equiv 1,\ e_{\lambda _{n+\alpha
,0}}(v;\mu )=e^{-2\pi iz_{\alpha }-\pi i\tau _{\alpha \alpha }},\ e_{\lambda
_{0,\beta }}(v;\mu )=\ e_{\lambda _{0,n+\beta }}(v;\mu )\equiv 1,\ \alpha
,\beta =1,...,n\}$(cf. (\ref{Multipliers})) and the similar multipliers for $%
\pi _{2}{}^{\ast }L_{0}$, we claim that a system of multipliers for $%
(Id\times \varphi _{L_{0}})^{\ast }P$ can be chosen to be 
\begin{equation}
\begin{cases}
e_{\lambda _{\alpha ,0}}(v;\mu )\equiv 1,\hspace*{10pt}e_{\lambda _{n+\alpha
,0}}(v;\mu )=e^{-2\pi i\mu _{\alpha }} \\ 
e_{\lambda _{0,\beta }}(v;\mu )\equiv 1,\hspace*{10pt}e_{\lambda _{0,n+\beta
}}(v;\mu )=e^{-2\pi iz_{\beta }}%
\end{cases}%
\alpha ,\beta =1,...,n.  \label{4.3}
\end{equation}%
\newline
\noindent Since (\ref{4.3}) satisfies (\ref{4.2}), it defines a holomorphic
line bundle $J$ on $M_{1}\times M_{2}$. To prove that $J\cong (Id\times
\varphi _{L_{0}})^{\ast }P$ the idea is to compare $i)$ the system of
multipliers for $(Id\times \varphi _{L_{0}})^{\ast }P|_{M\times \{\mu
\}}\cong \tau _{\mu }{}^{\ast }L_{0}\otimes L_{0}^{\ast }\cong P_{\varphi
_{L_{0}}(\mu )}\rightarrow M_{1}\times \{\mu \}$ given by $e_{\lambda
_{\alpha ,0}}(v)\equiv 1,\hspace*{4pt}e_{\lambda _{n+\alpha ,0}}(v)=e^{-2\pi
i\mu _{\alpha }}$ $(\alpha =1,...,n)$ with the system of multipliers for $%
J|_{M\times \{\mu \}}$, and\hspace*{4pt} $ii)$ the multipliers for the
trivial line bundle $(Id\times \varphi _{L_{0}})^{\ast }P|_{\{0\}\times M}$
by $e_{\lambda _{\alpha }}(v)=e_{\lambda _{n+\alpha ,0}}(v)\equiv 1$ ($%
\alpha =1,...,n)$ with the ones for $J|_{\{0\}\times M}$. The argument for
verifying $i)$, $ii)$ here is similar to \cite[Prop. 4.2]{CCT}.
\end{proof}

\noindent \hspace*{12pt} In the notation of Lemma \ref{Lemma 2.2} $\theta
_{m}(v;\mu )=\theta _{m}(v+\mu )$\ (with $m\in \mathfrak{M}$) and the metric 
$h(v;\mu )=$ $h_{L_{0}}(v+\mu )$ are considered as functions in the joint
variables $(v,\mu )\in V\times V$. From Proposition \ref{Prop4.2} and the
quasi-periodic properties of $\theta _{m}(\xi )$ and $h_{L_{0}}(\xi )$ it
holds the following global property of the family of functions $\{\theta
_{m}(v;\mu )\}_{m\in \mathfrak{M}}$ (cf. \cite[Prop. 4.3]{CCT}): \bigskip 

\begin{proposition}
\label{Prop4.3}For the above holomorphic line bundle $\widetilde{K}%
\rightarrow M_{1}\times M_{2}$, the functions $\theta _{m}(v;\mu )$, $m\in 
\mathfrak{M}$ just defined serves as a basis of holomorphic sections of $%
\widetilde{K}$, and $h(v;\mu )$ as a metric on $\widetilde{K}$, induces the
metric $h_{L_{\mu }}$ in Lemma \ref{Lemma 2.2} on the restriction $%
\widetilde{K}|_{M\times \left\{ \mu \right\} }$.
\end{proposition}

\noindent \hspace*{12pt} From the metrics $h(v;\mu )$ on $\widetilde{K}$, $%
h_{\pi _{1}^{\ast }L_{0}}(v;\mu )=e^{-2\pi \overset{n}{\underset{\alpha
,\beta =1}{\sum }}{W_{\alpha \beta }}\func{Im}(z_{\alpha })\func{Im}%
(z_{\beta })}$on $\pi _{1}^{\ast }L_{0}$ and $h_{\pi _{2}^{\ast
}L_{0}}(v;\mu )=e^{-2\pi \overset{n}{\underset{\alpha ,\beta =1}{\sum }}{%
W_{\alpha \beta }}\func{Im}(\mu _{\alpha })\func{Im}(\mu _{\beta })}$on $\pi
_{2}^{\ast }L_{0}$ (cf.\ $h_{L_{0}}(v)$ in Lemma \ref{Lemma 2.2}), we can
equip the line bundle $(Id\times \varphi _{L_{0}})^{\ast }P=(\pi _{1}^{\ast
}L_{0})^{\ast }\otimes \widetilde{K}\otimes (\pi _{2}^{\ast }L_{0})^{\ast }$
with an induced metric :%
\begin{equation}
h_{(Id\times \varphi _{L_{0}})^{\ast }P}(v;\mu )=(h_{\pi _{1}^{\ast
}L_{0}}(v;\mu ))^{-1}\text{ }h(v;\mu )\text{ }(h_{\pi _{2}^{\ast
}L_{0}}(v;\mu ))^{-1}=e^{-4\pi \overset{n}{\underset{\alpha ,\beta =1}{\sum }%
}{W_{\alpha \beta }}\func{Im}z_{\alpha }\func{Im}\mu _{\beta }}.\hspace*{20pt%
}  \label{hP}
\end{equation}%
\bigskip

\noindent \hspace*{12pt} For any metric $h$, we can calculate its curvature $%
\Theta $ by $\Theta =-\partial \overline{\partial }\log h$:

\begin{proposition}
\label{Prop 3.4} The curvature of the metric in $(\ref{hP})$ is%
\begin{equation}  \label{CuvatureP}
\Theta _{(Id\times \varphi _{L_{0}})^{\ast }P}(v;\mu )=\pi \overset{n}{%
\underset{\alpha ,\beta =1}{\sum }}{W_{\alpha \beta }}(dz_{\alpha }\wedge d%
\overline{\mu }_{\beta }+d\mu _{\alpha }\wedge d\overline{z}_{\beta })\text{
\ \ \ \ on }M_{1}\times M_{2}.\hspace*{80pt}
\end{equation}

\begin{proof}
By writing 
\begin{equation*}
h_{(Id\times \varphi _{L_{0}})^{\ast }P}(v;\mu )=e^{-4\pi \overset{n}{%
\underset{\alpha ,\beta =1}{\sum }}{W_{\alpha \beta }}\func{Im}(z_{\alpha })%
\func{Im}(\mu _{\beta })}=e^{\pi \overset{n}{\underset{\alpha ,\beta =1}{%
\sum }}{W}_{\alpha \beta }(z_{\alpha }-\overline{z}_{\alpha })(\mu _{\beta }-%
\overline{\mu }_{\beta })},
\end{equation*}

a straightfoward caculation gives%
\begin{eqnarray*}
\Theta _{(Id\times \varphi _{L_{0}})^{\ast }P}(v;\mu ) &=&-\partial 
\overline{\partial }\, \Big(\pi \overset{n}{\underset{\alpha ,\beta =1}{\sum 
}}{W}_{\alpha \beta }(z_{\alpha }-\overline{z}_{\alpha })(\mu _{\beta }-%
\overline{\mu }_{\beta })\Big) \\
&=&\pi \overset{n}{\underset{\alpha ,\beta =1}{\sum }}{W_{\alpha \beta }}%
(dz_{\alpha }\wedge d\overline{\mu }_{\beta }+d\mu _{\beta }\wedge d%
\overline{z}_{\alpha })=\pi \overset{n}{\underset{\alpha ,\beta =1}{\sum }}{%
W_{\alpha \beta }}(dz_{\alpha }\wedge d\overline{\mu }_{\beta }+d\mu
_{\alpha }\wedge d\overline{z}_{\beta })
\end{eqnarray*}

since $W=({W_{\alpha \beta }})$ is symmetric.
\end{proof}
\end{proposition}

\section{\textbf{Direct image bundle $K\rightarrow M_{2}$ with an }$L^{2}$%
\textbf{-metric}}

\noindent \hspace*{12pt} With the projections $\pi _{i}:M_{1}\times
M_{2}\rightarrow M_{i},i=1,2$, the direct image bundle $K$ for $\widetilde{K}
$ in Section 3 is given as follows (by projection formula on the third
factor of $\widetilde{K}$) 
\begin{equation}
K:=\pi _{2}{}_{\ast }\widetilde{K}=\pi _{2}{}_{\ast }(\pi _{1}^{\ast }L_{0}%
{\normalsize \otimes (Id\times \varphi _{L_{0}})^{\ast }P\otimes }\pi
_{2}^{\ast }L_{0})\cong \pi _{2}{}_{\ast }(\pi _{1}^{\ast }L_{0}{\normalsize %
\otimes (Id\times \varphi _{L_{0}})^{\ast }P)\otimes }L_{0}\rightarrow M_{2}.
\label{Kdirect}
\end{equation}

\begin{definition}
\label{Def5.1} Define a meric $($ $\ ,\ \ )_{h}$ on $K$ by the inner product
using $($ $\ ,\ \ )_{{h_{L_{\mu }}}}$on $K|_{\mu }=H^{0}(M,\widetilde{K}%
|_{M\times \{\mu \}}) : $ 
\begin{equation}
(\, \theta _{m}(v),\, \theta _{m^{\prime }}(v)\, )_{h_{L_{\mu}}}:=\int_{M} {%
h_{L_{\mu }}}(v)\, \theta _{m}(v)\, \overline{ \theta _{m^{\prime }}(v)}\, (%
\frac{i}{2})^{n} dz_{1} \wedge d\overline{z}_{1} \wedge ....\wedge dz_{n}
\wedge d\overline{z}_{n}.
\end{equation}
\end{definition}

\bigskip

\noindent \hspace*{12pt} The main lemma for our computations is as follows.

\begin{lemma}
\label{Lemma5.2} With the inner product $($ $\ ,\ \ )_{{h_{L_{\mu }}}},$ the
basis $\{\theta _{m}(v;\mu )\}_{m \in \mathfrak{M}}$ in Lemma \ref{Lemma 2.2}
consititute an orthogonal basis of $H^{0}(M,L_{\mu })$, with%
\begin{equation*}
(\,\theta _{m}(v;\mu ),\,\theta _{m}(v;\mu ))_{h_{L_{\mu }}}=\sqrt{\det (%
\frac{\func{Im}\tau _{\alpha \beta }}{2})}\left(\overset{n}{\underset{%
\alpha=1}{\Pi}}\delta_{\alpha} \right) e^{2\pi \overset{n}{\underset{\alpha
,\beta =1}{\sum }}\func{Im}\tau {_{\alpha \beta }}\frac{m_{\alpha }m_{\beta}%
}{\delta _{\alpha } \delta _{\beta }}}.
\end{equation*}
\end{lemma}

\begin{proof}
Write $z_{\alpha }=z_{\alpha 1}+i\, z_{\alpha2 }$, $\tau_{\alpha
\beta}=\tau_{\alpha \beta 1}+ i \tau_{\alpha \beta 2}$ $(\, z_{\alpha 1},
z_{\alpha 2}, \tau_{\alpha \beta 1}, \tau_{\alpha \beta 2}\in\mathbb{R}\, )$
and $dvol = dz_{11}dz_{12}...dz_{n1}dz_{n2}$. We have, via Definition 4.1
and Lemma 2.2, for $m, m^{\prime }\in \mathfrak{M}$,

\begin{equation}  \label{Immte}
I_{m m^{\prime }}:=(\,\theta _{m}(v;\mu ),\,\theta _{m^{\prime }}(v;\mu
))_{h_{L_{\mu }}}=\int_{M}h_{L_{\mu}}(v)\,\theta _{m}(v)\,\overline{\theta
_{m^{\prime }}(v)}\ dvol \hspace*{140pt}\vspace*{16pt}
\end{equation}
\begin{align*}
=\int_{M}e^{-2\pi \overset{n}{\underset{\alpha ,\beta =1}{\sum }}{W}_{\alpha
\beta }\left( z_{\alpha 2}+\mu_{\alpha 2} \right) \left( z_{\beta 2}
+\mu_{\beta 2}\right)}\,\sum\limits_{k,j\in \mathbb{Z}^{n}}\Big( e^{\pi i{%
\overset{n}{\underset{\alpha ,\beta =1}{\sum } }}k_{\alpha }k_{\beta }\tau
_{\alpha \beta }}e^{2\pi i{\overset{n}{\underset{\alpha ,\beta =1}{\sum }}}%
\tau _{\alpha \beta }\frac{m_{\alpha }}{\delta _{\alpha }}k_{\beta }}e^{2\pi
i{\overset{n}{\underset{\alpha =1}{\sum }}} \frac{(k_{\alpha }\delta
_{\alpha }+m_{\alpha })}{\delta _{\alpha }}(z_{\alpha }+\mu_{\alpha})}\, %
\Big) \\
\hspace*{40pt}\Big( e^{-\pi i{\overset{n}{\underset{\alpha ,\beta =1}{ \sum }%
}}j_{\alpha }j_{\beta }\overline{\tau} _{\alpha \beta }}e^{-2\pi i{\overset{n%
}{\underset{\alpha ,\beta =1}{\sum }}}\overline{\tau }_{\alpha \beta }\frac{%
m_{\alpha }^{\prime }}{\delta _{\alpha }}j_{\beta }}e^{-2\pi i{\overset{n}{%
\underset{\alpha =1}{\sum }}}\frac{(j_{\alpha }\delta _{\alpha }+m_{\alpha
}^{\prime })}{\delta _{\alpha }}(\overline{z}_{\alpha }+ \overline{\mu}%
_{\alpha})}\, \Big)\ dvol
\end{align*}
\begin{align*}
=\int_{M}e^{-2\pi \overset{n}{\underset{\alpha ,\beta =1}{\sum }}{W}_{\alpha
\beta } \left( z_{\alpha 2}+\mu_{\alpha 2} \right) \left( z_{\beta 2} +
\mu_{\beta 2}\right)}\hspace*{320pt} \\
\hspace*{40pt}\cdot \underset{j,k \in \mathbb{Z}^{n}}{\sum} e^{\pi i 
\underset{\alpha , \beta =1}{\overset{n}{\sum}} (k_{\alpha} k_{\beta}
\tau_{\alpha \beta}-j_{\alpha}j_{\beta} \overline{\tau}_{\alpha \beta})}e^{2
\pi i \overset{n}{\underset{\alpha ,\beta=1}{\sum}}(k_{\alpha} \frac{%
m_{\beta}}{\delta_{\beta}}\tau_{\alpha \beta}-j_{\alpha} \frac{m^{\prime
}_{\beta}}{\delta_{\beta}} \overline{\tau}_{\alpha \beta})}e^{2 \pi i\overset%
{n}{\underset{\alpha=1}{\sum}}\big[ (k_{\alpha} + \frac{m_{\alpha}}{\delta}%
)(z_{\alpha}+\mu_{\alpha})-(j_{\alpha}+\frac{m^{\prime }_{\alpha}}{%
\delta_{\alpha}})(\overline{z}_{\alpha}+\overline{\mu}_{\alpha})\big]}\ dvol
\end{align*}

\noindent We first show that $I_{mm^{\prime }}=0$ if $m\neq m^{\prime }$.
\noindent With the change of variables $z_{\alpha 1}=\delta _{\alpha
}\,t_{\alpha }+\overset{n}{\underset{l=1}{\sum }}\,\tau _{\alpha l1}\,t_{n+l}
$, $z_{\alpha 2}=\overset{n}{\underset{l=1}{\sum }}\tau _{\alpha l2}\,t_{n+l}
$,\ $\alpha =1,...,n$, we can change the domain of integration from $M$ to
the $2n$-dimensional unit cube $I^{2n}$:\newline

\begin{eqnarray}
&&I_{mm^{\prime }} =\int_{I^{2n}}e^{-2\pi \overset{n}{\underset{\alpha
,\beta =1}{\sum }}{W}_{\alpha \beta }\left( (\overset{n}{\underset{l=1}{\sum 
}}\tau _{\alpha l2}t_{n+\alpha })+\mu _{\alpha 2}\right) \left( (\overset{n}{%
\underset{l=1}{\sum }}\tau _{\alpha l2}t_{n+\beta })+\mu _{\beta 2}\right) }%
\hspace*{160pt}  \label{Imm'} \\
&&\hspace*{45pt}\cdot \underset{j,k\in \mathbb{Z}^{n}}{\sum }e^{\pi i%
\underset{\alpha ,\beta =1}{\overset{n}{\sum }}(k_{\alpha }k_{\beta }\tau
_{\alpha \beta }-j_{\alpha }j_{\beta }\overline{\tau }_{\alpha \beta
})}\cdot \ e^{2\pi i\overset{n}{\underset{\alpha ,\beta =1}{\sum }}%
(k_{\alpha }\frac{m_{\beta }}{\delta _{\beta }}\tau _{\alpha \beta
}-j_{\alpha }\frac{m_{\beta }^{\prime }}{\delta _{\beta }}\overline{\tau }%
_{\alpha \beta })}e^{2\pi i\overset{n}{\underset{\alpha =1}{\sum }}\big[%
(k_{\alpha }+\frac{m_{\alpha }}{\delta _{\alpha }}-j_{\alpha }-\frac{%
m_{\alpha }^{\prime }}{\delta _{\alpha }})(\delta _{\alpha }t_{\alpha })\big]%
}  \notag \\
&&\hspace*{170pt}\cdot \ e^{\pi i\big[(k_{\alpha }+\frac{m_{\alpha }}{\delta
_{\alpha }})(\overset{n}{\underset{l=1}{\sum }}\tau _{\alpha l}t_{n+l}+\mu
_{\alpha })-(j_{\alpha }+\frac{m_{\alpha }^{\prime }}{\delta _{\alpha }})(%
\overset{n}{\underset{l=1}{\sum }}\overline{\tau }_{\alpha l}t_{n+l}+%
\overline{\mu }_{\alpha })\big]}\ \mathcal{J}\ dt_{1}...dt_{2n}  \notag
\end{eqnarray}%
\noindent where $\mathcal{J}=\Big(\overset{n}{\underset{\alpha =1}{\Pi }}%
\delta _{\alpha }\Big)\det (\tau _{\alpha \beta 2})$ is the Jacobian.%
\vspace*{8pt}\newline
\noindent \hspace*{12pt} Checking the terms in the integrand related to $%
t_{1},...,t_{n}$ we have the third term in the second line above:%
\begin{equation}
\sum\limits_{k,j\in \mathbb{Z}^{n}}\int_{I^{n}}e^{\ \overset{n}{\underset{%
\alpha =1}{\sum }}2\pi i(k_{\alpha }\delta _{\alpha }+m_{\alpha }-j_{\alpha
}\delta _{\alpha }-m_{\alpha }^{\prime })\,t_{\alpha }}\ dt_{1}...dt_{n}.%
\hspace*{100pt}
\end{equation}%
\noindent Now $0\leq m_{\alpha },m_{\alpha }^{\prime }<\delta _{\alpha }$
and $m_{\alpha },m_{\alpha }^{\prime }\in 0\cup \mathbb{N}$. If $j\neq k$
then $|(k_{\alpha }-j_{\alpha })\delta _{\alpha }|\geq \delta _{\alpha }$
for some $\alpha =1,...,n$, $k_{\alpha }\delta _{\alpha }+m_{\alpha
}-j_{\alpha }\delta _{\alpha }-m_{\alpha }^{\prime }\in \mathbb{Z}\setminus
\{0\}$ since $|m_{\alpha }-m_{\alpha }^{\prime }|<\delta _{\alpha }$, and
so, for this $\alpha $ 
\begin{equation}
\int_{I}e^{\ 2\pi i(k_{\alpha }\delta _{\alpha }+m_{\alpha }-j_{\alpha
}\delta _{\alpha }-m_{\alpha }^{\prime })\,t_{\alpha }}\ dt_{\alpha }=0.%
\hspace*{100pt}  \label{4.5}
\end{equation}%
It remains to consider the $j=k$ terms. We have, for all $k\in \mathbb{Z}%
^{n} $ 
\begin{equation}
\int_{I}e^{\ 2\pi i(k_{\alpha }\delta _{\alpha }+m_{\alpha }-j_{\alpha
}\delta _{\alpha }-m_{\alpha }^{\prime })\,t_{\alpha }}\ dt_{\alpha
}=\int_{I}e^{\ 2\pi i(m_{\alpha }-m_{\alpha }^{\prime })\,t_{\alpha }}\
dt_{\alpha }.\hspace*{50pt}  \label{j=k}
\end{equation}%
By using $(\ref{j=k})$ for $(\ref{Imm'})$ one sees that the integral $%
I_{mm^{\prime }}$ vanishes if $m\neq m^{\prime }$, proving the orthogonal
property for $\{\theta _{m}(v;\mu ))\}_{m\in \mathfrak{M}}$.\newline

\noindent \hspace*{12pt} The linear independence of $\{\theta _{m}(v;\mu
)\}_{m\in \mathfrak{M}}$ in $H^{0}(M, L _{\mu})$ as needed in the proof of
Lemma \ref{Lemma 2.1} follows from the above orthogonal property.\newline

\noindent For $m=m^{\prime }\in \mathfrak{M}$ we calculate, using $(\ref%
{Immte})$ again and setting $m=m^{\prime }$ in the third equality below 
\begin{eqnarray}  \label{Imteg}
I_{m} :=(\,\theta _{m}(v;\mu ),\,\theta _{m}(v;\mu ))_{h_{L_{\mu
}}}=\int_{M}h_{L_{\mu }}(v)\,\theta _{m}(v;\mu )\,\overline{\theta
_{m}(v;\mu )}\, dvol\hspace*{110pt}  \notag \\
=\int_{M}e^{-2\pi \overset{n}{\underset{\alpha ,\beta =1}{\sum }}{W_{\alpha
\beta }}(z_{\alpha 2}+\mu _{\alpha 2})(z_{\beta 2}+\mu _{\beta 2}) }\hspace*{%
246pt}  \notag \\
\cdot \sum\limits_{k,j\in \mathbb{Z}^{n}}\Big( e^{\pi i{\overset{n}{\underset%
{\alpha ,\beta =1}{\sum }}}k_{\alpha }k_{\beta }\tau _{\alpha \beta
}}e^{2\pi i{\overset{n}{\underset{\alpha ,\beta =1}{\sum }}}\tau _{\alpha
\beta }\frac{m_{\alpha }}{\delta _{\alpha }}k_{\beta }}e^{2\pi i{\overset{n}{%
\underset{\alpha =1}{\sum }}}\frac{(k_{\alpha }\delta _{\alpha }+m_{\alpha })%
}{\delta _{\alpha }}(z_{\alpha }+\mu _{\alpha })}\, \Big) \hspace*{80pt} 
\notag \\
\cdot \Big( e^{-\pi i{\overset{n}{\underset{\alpha ,\beta =1}{\sum }}}%
j_{\alpha }j_{\beta }\overline{\tau}_{\alpha \beta }}e^{-2\pi i{\overset{n}{%
\underset{\alpha ,\beta =1}{\sum }}}\overline{\tau} _{\alpha \beta }\frac{%
m_{\alpha }^{\prime }}{\delta _{\alpha }}j_{\beta }}e^{-2\pi i{\overset{n}{%
\underset{\alpha =1}{\sum }}}\frac{(j_{\alpha }\delta _{\alpha}+m_{\alpha
}^{\prime })}{\delta _{\alpha }}(\overline{z}_{\alpha }+ \overline{\mu }%
_{\alpha })}\, \Big)\, dvol \\
\overset{(\text{cf}.\, (\ref{4.5}))}{=} \int_{M}\,\underset{k\in \mathbb{Z}%
^{n}}{\sum}\Big(e^{-2 \pi \, {\overset{n}{\underset{\alpha ,\beta =1}{\sum }}%
}k_{\alpha }k_{\beta }\tau _{\alpha \beta 2}}\, \Big)\hspace*{250pt}  \notag
\\
\cdot \Big( e^{{\scriptsize -2\pi \big[ \overset{n}{\underset{\alpha ,\beta
=1}{\sum }}{W_{\alpha \beta }}(z_{\alpha 2}+\mu_{\alpha 2})(z_{\beta 2}+\mu
_{\beta 2})+2{\overset{n}{\underset{\alpha ,\beta =1}{\sum }}}\tau _{\alpha
\beta 2}\frac{m_{\alpha }}{\delta _{\alpha }}k_{\beta }+2{\overset{n}{%
\underset{\alpha =1}{\sum }}}\frac{(k_{\alpha }\delta _{\alpha }+m_{\alpha })%
}{\delta _{\alpha }}(z_{\alpha 2}+\mu _{\alpha 2})\big] }}\, \Big) dvol. 
\notag
\end{eqnarray}

\noindent It is slightly tedious but straightforward for (\ref{Imteg}) to
complete the square in the following form: (here identities such as ${%
\overset{n}{\underset{\alpha =1}{\sum }}}\frac{k_{\alpha }\delta _{\alpha
}+m_{\alpha }}{\delta _{\alpha }}(z_{\alpha 2}+\mu _{\alpha 2})$ $=$ ${%
\overset{n}{\underset{\alpha ,\beta ,\eta =1}{\sum }}}W_{\alpha \beta
}\,(z_{\alpha 2}+\mu _{\alpha 2})\,\tau _{\beta \eta 2}\,\frac{k_{\eta
}\delta _{\eta }+m_{\eta }}{\delta _{\eta }}$ since ${\overset{n}{\underset{%
\beta =1}{\sum }}}W_{\alpha \beta }\,\tau _{\beta \eta 2}$ $=$ $\delta
_{\eta }^{\alpha }$ are useful)

\begin{eqnarray}  \label{Im}
I_{m}=e^{2\pi {\overset{n}{\underset{\alpha ,\beta =1}{\sum }}}\tau _{\alpha
\beta 2 }\frac{m_{\alpha m_{\beta}}}{\delta _{\alpha } \delta_{\beta}}}%
\hspace*{320pt}  \notag \\
\cdot \int_{M}\,\sum\limits_{k\in \mathbb{Z}^{n}}e^{-2\pi \big[ \overset{n}{%
\underset{\alpha ,\beta =1}{\sum }}{W_{\alpha \beta }}( z_{\alpha 2}+\mu
_{\alpha 2}+{\overset{n}{\underset{\gamma =1}{\sum }}}\tau _{\alpha \gamma 2}%
\frac{(k_{\gamma }\delta _{\gamma }+m_{\gamma })}{\delta _{\gamma }}) (
z_{\beta 2}+\mu _{\beta2 }+{\overset{n}{\underset{\eta =1}{\sum }}}\tau
_{\beta \eta 2}\frac{(k_{\eta }\delta_{\eta }+m_{\eta })}{\delta _{\eta }}) %
\big] }dvol.
\end{eqnarray}

\noindent With the same change of variables $z_{\alpha 1}=\delta_{\alpha}\,
t_{\alpha}+\overset{n}{\underset{l=1}{\sum }}\tau_{\alpha l 1}\, t_{n+l}$, $%
z_{\alpha 2}=\overset{n}{\underset{l=1}{\sum }}\tau_{\alpha l 2}\, t_{n+l}$
as before, we are now integrating\footnote{%
It is not absolutely necessary to find this explicit form. If one uses $(\ref%
{Im})$ and regards it as an \textit{implicit function} in terms of $t_{1},
t_{2},...,t_{2n}$, one first observes that the integration involving $%
dt_{1}dt_{2}...dt_{n}$ is easily obtained as that in $(\ref{Im2})$ and then
that the complete integration can still be done over $\mathbb{R}^{n}$ in a
way similar to $(\ref{B4Gauss})$, with the previous implicit functions.
Being on $\mathbb{R}^{n}$, one returns now to the original variables; in
this way one soon obtains the same result via the Gaussian integration $(\ref%
{Gaussian})$ applied to the quadratic term in the bracket $[...]$ of $(\ref%
{Im})$.} over the $2n$-dimensional unit cube $I^{2n} $:

\begin{eqnarray}  \label{4.9}
I_{m}=e^{2\pi {\overset{n}{\underset{\alpha ,\beta =1}{\sum }}}\tau
_{\alpha\beta 2 }\frac{m_{\alpha } m_{\beta}}{\delta _{\alpha }
\delta_{\beta}}}\hspace*{330pt}  \notag \\
\hspace*{40pt}\cdot \int_{I^{2n}}\,\sum\limits_{k\in \mathbb{Z}^{n}}e^{-2\pi %
\big[\overset{n}{\underset{\alpha ,\beta =1}{\sum }}{\tau_{\alpha \beta 2}}%
(t_{n+\alpha}+k_{\alpha} +\frac{m_{\alpha}}{\delta_{\alpha}}+\overset{n}{%
\underset{l=1}{\sum}}W_{\alpha l} \mu _{l 2}) (t_{n+\beta}+k_{\beta} +\frac{%
m_{\beta}}{\delta_{\beta}}+\overset{n}{\underset{\gamma=1}{\sum}}W_{\beta
\gamma} \mu _{\gamma 2}) \big] }\, \mathcal{J}\ dt_{1}...dt_{2n}.
\end{eqnarray}

\bigskip \noindent \hspace{12pt} To simplify the notation, we let $%
G(t_{n+1},...,t_{2n})=e^{-2 \pi \overset{n}{\underset{\alpha , \beta =1}{\sum%
}} \tau_{\alpha \beta 2} t_{n+\alpha }t_{n+\beta }}$ and $%
p=(p_{1},...,p_{n}) $ where $p_{\alpha} = \frac{m_{\alpha}}{\delta_{\alpha}}+%
\overset{n}{\underset{l=1}{\sum}}W_{\alpha l} \mu _{l 2}$, and rewrite $(\ref%
{4.9})$ as

\begin{align}  \label{Im2}
\hspace*{12pt} I_{m}=e^{2\pi {\overset{n}{\underset{\alpha ,\beta =1}{\sum }}
}\tau _{\alpha \beta 2}\frac{m_{\alpha } m_{\beta}}{\delta _{\alpha }
\delta_{\beta}}}\hspace*{330pt}  \notag \\
\cdot \mathcal{J} \underset{{k\in \mathbb{Z}^{n}}}{\sum}\, \big( %
\int_{I^{n}}\, 1\, dt_{1}...dt_{n}\, \big) \int_{I^{n}}
G(t_{n+1}+k_{1}+p_{1},...,t_{2n}+k_{n}+p_{n})\ dt_{n+1}...dt_{2n}.\vspace*{%
12pt}
\end{align}

\bigskip \noindent By Sublemma \ref{sublem} below we have 
\begin{equation}  \label{B4Gauss}
I_{m}=e^{2\pi {\overset{n}{\underset{\alpha ,\beta =1}{\sum }}}\tau _{\alpha
\beta 2}\frac{m_{\alpha } m_{\beta}}{\delta _{\alpha }\delta_{\beta}}}%
\mathcal{J} \int_{\mathbb{R}^{n}}e^{-\overset{n}{\underset{\alpha ,\beta =1}{%
\sum }}2 \pi \, \tau _{\alpha \beta 2 }\, t_{n+\alpha }t_{n+\beta
}}dt_{n+1}...dt_{2n}.\hspace*{80pt}
\end{equation}
Applying to $(\ref{B4Gauss})$ the Gaussian integral 
\begin{equation}  \label{Gaussian}
\int_{%
\mathbb{R}
^{n}}e^{-X^{T}AX}dt_{n+1}...dt_{2n}=\frac{(\sqrt{\pi })^{n}}{\sqrt{\det A}}%
\hspace*{100pt}
\end{equation}%
where $A$ is a positive definite $n\times n$ matrix, we get the result : 
\begin{equation}
I_{m}=\sqrt{\det (\frac{\func{Im}\tau _{\alpha \beta }}{2})}\,\big( \overset{%
n}{\underset{\alpha=1}{\Pi}}\delta_{\alpha} \big)\, e^{2\pi \overset{n}{%
\underset{\alpha ,\beta =1}{\sum }}\func{Im}\tau_{\alpha \beta }\frac{%
m_{\alpha }m_{\beta}}{\delta _{\alpha }\delta_{\beta}}}.\hspace*{80pt}
\end{equation}
\end{proof}

\begin{sublemma}
\label{sublem} Let $f(x):\mathbb{R}^{n}\rightarrow \mathbb{R}$ be an
integrable function, $p\in \mathbb{R}^{n}$ any fixed point. Then 
\begin{equation*}
\underset{k\in \mathbb{Z}^{n}}{\sum }\,\int_{I^{n}}\,f(x+k+p)\ dx=\int_{%
\mathbb{R}^{n}}\,f(x+p)\ dx=\int_{\mathbb{R}^{n}}\,f(x)\ dx.\hspace*{60pt}%
\vspace*{16pt}
\end{equation*}
\end{sublemma}

\noindent \hspace*{12pt} By Lemma \ref{Lemma5.2}, the value of $(\theta
_{m}(z;\mu ),\theta _{m}(z;\mu ))_{h_{L_{\mu }}}$ in Definition \ref{Def5.1}
is independent of $\mu $, giving the first part of the following theorem.
The second part of it follows from the first statement, Proposition $\ref%
{Prop4.3}$, and Lemma \ref{Lemma5.2}:\vspace*{4pt}

\begin{theorem}
\label{theo4.3} $(1)$ On $K$, the curvature tensor of the metric $(\ ,\
)_{h} $ defined in Definition \ref{Def5.1} is identically zero. $(2)$ $K$
splits holomorphically into a direct sum of holomorphically trivial line
bundle $K=\underset{m\in \mathfrak{M}}{\oplus }K_{m}$ where each $K_{m}$ has
the canonical section identified as $\theta _{m}$ in Lemma \ref{Lemma5.2}.
Note that for any $\mu $ $\in $ $M_{2}(=M)$ $\theta _{m}(v;\mu )$ $\in $ $%
(K_{m})_{\mu }$ is nontrivial although it has zeros along $v,$ so that $%
\theta _{m}(v;\mu )$ is nowhere vanishing as a section of $K_{m}.$
\end{theorem}

\section{\textbf{Canonical Connection on the Poincar\'{e} line bundle and
Proof of Theorem \protect\ref{thm1.2}}}

\noindent \hspace*{12pt} In this section, we view the Abelian variety $M$ as
a real $2n$-dimensional manifold and intoduce a geometric description of the
Poincar\'e line bundle with a connection on it. Similar to \cite{CCT}, we
first define a line bundle $\mathscr{P} \rightarrow M \times M^{*}$ with a
connection and then identify it with the Poincar\'e line bundle $P
\rightarrow M \times \widehat{M}$. We follow the treatment in \cite{CCT} and 
\cite{DK} and generalize it to the $2n$-dimensional abelian variety $M$.%
\newline

\noindent \hspace*{12pt} To begin with, we write $V$ $\cong $ $\mathbb{R}%
^{2n}$ and $M$ $=$ $V/\Lambda $ where $\Lambda $ $=$ $\mathbb{Z}\{\lambda
_{1},...,\lambda _{2n}\}$ and in some standard coordinates of $\mathbb{R}%
^{2n}$ (see Notation \ref{N-2-1}), $\lambda _{1}=(\delta _{1},0,...,0,0)$, $%
\lambda _{2}=(0,0,\delta _{2},0,...,0,0)$,...,$\lambda _{n}=(0,0,....,\delta
_{n},0)$, and $\lambda _{n+\alpha }=\big(\tau _{\alpha 11},\tau _{\alpha
12},...,\tau _{\alpha n1},\tau _{\alpha n2}\big)$ for $\alpha =1,...,n$;
these $\delta $'s and $\tau $'s ($\tau _{\alpha 1}=\tau _{\alpha 11}+i\tau
_{\alpha 12}$) are from the period matrix $\Omega $ in Section 2. Let $%
\Lambda ^{\ast }=\mathbb{Z}\{dx_{1},...,dx_{2n}\}$ be the dual lattice of $%
\Lambda $ (so $\int_{\lambda _{i}}dx_{j}=\delta _{ij}$). Let $V^{\ast
}:=Hom_{\mathbb{R}}(V,\mathbb{R})=\{\xi _{1}dx_{1}+...+\xi _{2n}dx_{2n}\mid
\xi _{i}\in \mathbb{R},i=1,...,2n\}$. We define $M^{\ast }:=V^{\ast }/2\pi
\Lambda ^{\ast }$ and write $[\xi ]$ for the equivalence class of $\xi $ in $%
M^{\ast }$.\newline

\noindent \hspace*{12pt} Let $\underline{\mathbb{C}}|_{V}:V\times \mathbb{C}%
\rightarrow V$ be the trivial complex line bundle over $V$. An element $\xi
=\sum \xi _{i}dx_{i}\in V^{\ast }$ gives rise to a character $\chi _{\xi
}:\Lambda \rightarrow U(1)$ by $\chi _{\xi }(\lambda ):=e^{-i<\xi ,\lambda
>}=e^{-i\xi (\lambda )}$. The set $\Lambda $ acts on $\underline{\mathbb{C}}%
|_{V}$ by $\lambda \circ (x,\sigma ):=(x+\lambda ,\chi _{\xi }(\lambda )\,
\sigma )$. The natural horizontal foliation in $\underline{\mathbb{C}}|_{V}$
is preserved by such actions and thus descends to a flat connection $d$ on
the quotient $\underline{\mathbb{C}}|_{M}=\underline{\mathbb{C}}|_{V}$ $/$ $%
\Lambda $ over $M$. For $\xi \in V^{\ast }$, one defines a flat $U(1)$
connection on the complex line bundle $\underline{\mathbb{C}}|_{M}$ over $M$
by $\nabla ^{\xi }:=d+i\xi .$\newline

\noindent \hspace*{12pt} The gauge equivalence classes of flat line bundles
on $M$ are parametrized by $M^{\ast }$. We write $\mathcal{L}_{[\xi
]}:=\left( \underline{\mathbb{C}}|_{M},\nabla ^{\xi }\right) $ for the flat
line bundle on $M$ corresponding to the connection $\nabla ^{\xi },\xi \in
V^{\ast }$. With the connection $\nabla ^{\xi }$, the parallel transport
along the loop is given by $\chi _{\xi }$ above. Dually, for any $x\in V$ we
define a character $\chi _{x}:2\pi \Lambda ^{\ast }\rightarrow U(1)$ by $%
\chi _{x}(2\pi \mu ):=e^{-2\pi i<\mu ,x>}$ and get flat line bundles $%
\mathcal{L}_{[x]}$ over $M^{\ast }$ with the parallel transport $\chi _{x}$.
Similar to the $2$-dimensional case in \cite{CCT} and $4$-dimensional case
in \cite{DK}, we have the following lemma for the $2n$-dimensional $M$:

\begin{lemma}
\label{Pcal}There is a complex line bundle $\mathscr{P}\rightarrow M\times
M^{\ast }$with a unitary connection such that the restriction of $\mathscr{P}
$ to each $M_{[\xi ]}=M\times \{[\xi ]\}$ is isomorphic (as a line bundle
with connection) to $\mathcal{L}_{[\xi ]}$ and the restriction to each $%
M_{[x]}^{\ast }=\{[x]\}\times M^{\ast }$ is isomorphic to $\mathcal{L}_{[x]}$%
.
\end{lemma}

\noindent \hspace*{12pt} In essence, lifting the action of $2\pi \Lambda
^{\ast }$ on $M\times V^{\ast }$ to the trivial line bundle $\underline{%
\mathbb{C}}|_{M\times V^{\ast }}$ over $M \times V^{*}$ endowed with the
connection one form $\mathbb{A}=i\xi =i \overset{2n}{\underset{k=1}{\sum}}
\xi _{k}\, dx_{k}\ (\xi \in V^{\ast })$ by $2\pi \nu \circ (x,\xi ,\sigma
):=(x,\xi +2\pi \nu ,e^{-2\pi i<\nu ,x>}\sigma )$ for $\nu \in \Lambda
^{\ast }$, one sees that this action preserves the connection $d+\mathbb{A}$
on $\underline{\mathbb{C}}|_{M\times V^{\ast }}$ and hence induces a
connection on the line bundle $\mathscr{P}:=\underline{\mathbb{C}}|_{M\times
V^{\ast }}/2\pi \Lambda ^{\ast }\rightarrow M\times M^{\ast }$, denoted by $%
\nabla ^{\mathscr{P}}$. We sometimes write 
\begin{equation}
\nabla ^{\mathscr{P}}=d+\mathbb{A}\hspace*{100pt}  \label{d+A}
\end{equation}%
by abuse of notation. The connection $\nabla ^{\mathscr{P}}$ has the
curvature $\Theta _{\nabla ^{\mathscr{P}}}=d\mathbb{A}=i\,\underset{k=1}{%
\overset{2n}{\sum }}\,d\xi _{k}\wedge dx_{k}$.

\noindent \hspace*{12pt} If we define a metric on the trivial line bundle $%
\underline{\mathbb{C}}|_{M\times V^{\ast }}$ by $<(x,\xi ,\sigma
_{1}),(x,\xi ,\sigma _{2})>:=\sigma _{1}\cdot \overline{\sigma _{2}}$, this
metric is preserved by the action of $2\pi \Lambda ^{\ast }$, and hence
induces a metric $h_{\mathscr{P}}$ on $\mathscr{P}$. One sees that $\nabla ^{%
\mathscr{P}}$ is compatible with $h_{\mathscr{P}}$ as required in Lemma \ref%
{Pcal}.\vspace*{8pt}

\noindent \hspace{12pt} Before we identify the line bundle $\mathscr{P}$ of
lemma \ref{Pcal} with the Poincar\'{e} line bundle $P$, we should first have
an isomorphism from $\widehat{M}$ to $M^{\ast }$, whose explicit form will
be needed in a number of discussions later:

\begin{lemma}
\label{Iso}One has an isomorphism 
\begin{equation*}
Iso:\widehat{M}\overset{\sim }{\longrightarrow }M^{\ast }.\hspace*{180pt}
\end{equation*}
\end{lemma}

\begin{proof}
$\widehat{M}=\{\overset{n}{\underset{\alpha =1}{\sum }}\,\widehat{\mu }%
_{\alpha }v_{\alpha }^{\ast }\ |\ \widehat{\mu }_{\alpha }\in \mathbb{C}\}/%
\overline{\Lambda }^{\ast }$ \big(where $\widehat{\mu }_{\alpha }=\widehat{%
\mu }_{\alpha 1}+i\widehat{\mu }_{\alpha 2}$ and $\overline{\Lambda }%
^{\,\ast }=\{\overset{2n}{\underset{i=1}{\sum }}n_{i}dx_{i}^{\ast }\ |\
n_{i}\in \mathbb{Z}\}$ in the notation of Section 2\big)and $M^{\ast }=\{%
\overset{2n}{\underset{i=1}{\sum }}\,\xi _{i}dx_{i}^{\ast }\ |\ \xi _{i}\in 
\mathbb{R}
\}/2\pi \Lambda ^{\ast }$. The group isomorphism $Iso:\widehat{M}\rightarrow
M^{\ast }$ sends $dx_{i}^{\ast }$ to $2\pi dx_{i},$ $i=1,...,2n$ and
satisfies 
\begin{equation}
\begin{cases}
\begin{pmatrix}
\xi _{1} \\ 
\vdots \\ 
\xi _{n}%
\end{pmatrix}%
=2\pi \delta _{n}^{-1}\Delta _{\delta }W\Delta _{\delta }%
\begin{pmatrix}
\widehat{\mu }_{12} \\ 
\vdots \\ 
\widehat{\mu }_{n2}%
\end{pmatrix}%
\vspace*{8pt} \\ 
\begin{pmatrix}
\xi _{n+1} \\ 
\vdots \\ 
\xi _{2n}%
\end{pmatrix}%
=-2\pi \delta _{n}^{-1}\Delta _{\delta }%
\begin{pmatrix}
\widehat{\mu }_{11} \\ 
\vdots \\ 
\widehat{\mu }_{n1}%
\end{pmatrix}%
+2\pi \delta _{n}^{-1}(\func{Re}Z)W\Delta _{\delta }%
\begin{pmatrix}
\widehat{\mu }_{12} \\ 
\vdots \\ 
\widehat{\mu }_{n2}%
\end{pmatrix}%
.%
\end{cases}
\label{mutoxi}
\end{equation}%
Here $(\Delta _{\delta }|Z)$ is the period matrix $\Omega $ of $\Lambda
\subseteq V$ and $W$ is the inverse matrix of $\func{Im}Z$. Equivalently 
\begin{equation}
\begin{pmatrix}
\widehat{\mu }_{1} \\ 
\vdots \\ 
\widehat{\mu }_{n}%
\end{pmatrix}%
=(\frac{1}{2\pi }\delta _{n}\Delta _{\delta }^{-1}Z\Delta _{\delta }^{-1},%
\frac{-1}{2\pi }\delta _{n}\Delta _{\delta }^{-1})%
\begin{pmatrix}
\xi _{1} \\ 
\vdots \\ 
\xi _{2n}%
\end{pmatrix}%
.\hspace*{40pt}  \label{xitomu}
\end{equation}

\noindent In particular, for $n=1$, the isomorphism reduces to ([1, Lemma
7.1]) 
\begin{equation}
\begin{cases}
\xi_{1}= \frac{2 \pi \delta}{ \tau_{2}}\, \func{Im}(\widehat{\mu}) \\ 
\xi_{2}= \frac{2 \pi}{\tau_{2}}\, \left(\tau_{1} \func{Im}(\widehat{\mu})-
\tau_{2} \func{Re}(\widehat{\mu})\right)%
\end{cases}
\mbox{ and\hspace*{16pt}} 
\begin{cases}
\func{Re}(\widehat{\mu}) = \frac{-1}{2 \pi}\, \xi_{2} + \frac{1}{2 \pi} 
\frac{\tau_{1}}{\delta}\, \xi_{1} \\ 
\func{Im}(\widehat{\mu})= \frac{1}{2 \pi} \frac{\tau_{2}}{\delta}\, \xi_{1}%
\end{cases}%
\end{equation}

\noindent where $\Delta_{\delta}=\delta \in \mathbb{N}$ and $%
Z=\tau=\tau_{1}+i\tau_{2} \in \mathbb{C}$.
\end{proof}

\noindent \hspace*{12pt} We endow $M^{\ast }$ with the complex structure
inherited from $\widehat{M}$. Recall the Poincar\'{e} line bundle $%
P\rightarrow M\times \widehat{M}$ of Lemma \ref{Poincare}. To compare $P$
and $\mathscr{P}$, we first show that the global connection $\nabla ^{%
\mathscr{P}}$ on $\mathscr{P}$ (\ref{d+A}) gives a holomorphic structure on $%
\mathscr{P}$ (where this $M$ has been restored as an Abelian variety as
considered in the beginning of this section).\newline

\noindent \hspace*{12pt} For the holomorphic structure on $\mathscr{P}$,
write $\widetilde{Iso}:=(Id,Iso):M\times \widehat{M}\rightarrow M\times
M^{\ast }$ with the isomorphism $Iso:\widehat{M}\rightarrow M^{\ast }$ in
Lemma \ref{Iso}, and form the pull-back bundle $\widetilde{Iso}^{\ast }%
\mathscr{P}$ equipped with the metric $\widetilde{Iso}^{\ast }h_{\mathscr{P}}
$ and the connection $\widetilde{Iso}^{\ast }\nabla ^{\mathscr{P}}$. By (\ref%
{d+A}) $\nabla ^{\mathscr{P}}=d+i\xi ,\xi \in V^{\ast }$, one has by using $(%
\ref{realcomplex})$ and $(\ref{mutoxi})$ in the proof of Lemma \ref{Iso}
that 
\begin{equation}
\widetilde{Iso}^{\ast }\nabla ^{\mathscr{P}}=d+\mathbb{A}=d+\pi \overset{n}{%
\underset{\alpha ,\beta =1}{\sum }}\,W_{\alpha \beta }(\,\,\frac{\delta
_{\beta }}{\delta _{n}}\widehat{\mu }_{\beta }d\overline{z}_{\alpha }-\frac{%
\delta _{\alpha }}{\delta _{n}}\overline{\widehat{\mu }_{\alpha }}dz_{\beta
}\,)  \label{5.5}
\end{equation}%
and its curvature 
\begin{equation}
\Theta _{\widetilde{Iso}^{\ast }\nabla ^{\mathscr{P}}}(=d\mathbb{A})=\pi 
\overset{n}{\underset{\alpha ,\beta =1}{\sum }}\,W_{\alpha \beta }(\,\,\frac{%
\delta _{\beta }}{\delta _{n}}d\widehat{\mu }_{\beta }\wedge d\overline{z}%
_{\alpha }-\frac{\delta _{\alpha }}{\delta _{n}}d\overline{\widehat{\mu }%
_{\alpha }}\wedge dz_{\beta }\,),  \label{dA}
\end{equation}%
giving rise to a holomorphic line bundle structure on $\widetilde{Iso}^{\ast
}\mathscr{P}$ since $(\ref{dA})$ is of type $(1,1)$ (cf. \cite{CCT}). This
holomorphic structure induces a holomorphic structure on $\mathscr{P}$ via
the isomorphism $\widetilde{Iso}^{-1}$. We can now identify $P$ and $%
\mathscr{P}$ by showing that $P\cong \widetilde{Iso}^{\ast }\mathscr{P}$:

\begin{proposition}
\label{P to P} With the notations above, let $P\rightarrow M\times \widehat{M%
}$ be the Poincar\'{e} line bundle of Lemma \ref{Poincare}, and $\mathscr{P}%
\rightarrow M\times M^{\ast }$ of Lemma \ref{Pcal} with the holomorphic
structure just given. Then $P\cong \widetilde{Iso}^{\ast }\mathscr{P}$.
\end{proposition}

\begin{proof}
To prove that $P\cong \widetilde{Iso}^{\ast }\mathscr{P}$, by Lemma \ref%
{Poincare} it suffices to show that $\mathscr{P}\rightarrow M\times M^{\ast
} $ satisfies the following two properties : $i)$ for any $[\xi ]\in M^{\ast
} $, the line bundle $\mathcal{L}_{[\xi]} \cong \mathscr{P}|_{M \times
\{[\xi]\}}$ is holomorphically isomorphic to $P_{Iso^{-1}[\xi]}=P_{\widehat{%
\mu}}$ and $ii)$ $\mathscr{P}|_{\{0\} \times M^{*}}$ is holomorphically
trivial on $\{0\} \times M^{*}$.\newline

\noindent \hspace*{12pt} To prove $i)$ of the claim above, from the action $%
\lambda \circ (x,\sigma )=(x+\lambda ,e^{-i\,<\xi ,\lambda >}\,\sigma )$
given in the beginning of this section, the holonomy transforms the basis $%
\{\lambda _{1},...,\lambda _{2n}\}$ by $\chi _{\xi }(\lambda
_{j})=e^{-i\,<\xi ,\lambda _{j}>},\ j=1,...,2n$. The multipliers for $%
\mathcal{L}_{[\xi ]}$ are thus\footnote{%
We follow the convention of \cite[proof of Lemma in page 310]{GH} that the
multipliers transform the ``coefficient-part" rather than the ``basis-part",
and so they look inverse to holonomy transforms above.} 
\begin{equation}
e_{\lambda _{j}}(v)=e^{i\xi _{j}},\hspace{10pt}j=1,...,2n.\hspace*{234pt}
\label{multmathcalL}
\end{equation}%
\noindent On the other hand, we derive from $(\ref{2.3}),(\ref{theta section}%
)$ and (\ref{complex linear}) that the multipliers for $P_{\widehat{\mu }}$
with $\widehat{\mu }=\underset{\alpha }{\sum }\widehat{\mu }_{\alpha
}v_{\alpha }^{\ast }$ can be given by 
\begin{equation}
\begin{cases}
e_{\lambda _{\alpha }}(v)=1 \\ 
e_{\lambda _{n+\alpha }}(v)=e^{-2\pi i\,\mu _{\alpha }}=e^{-2\pi i\frac{%
\delta _{\alpha }}{\delta _{n}}\widehat{\mu }_{\alpha }},%
\end{cases}%
\alpha =1,...,n.\hspace*{170pt}  \label{multP}
\end{equation}

\noindent To match the two sets of multipliers $(\ref{multmathcalL})$ and $(%
\ref{multP} )$, let $L_{\Delta, \xi}=\mathcal{L}_{[\xi]} \otimes P^{*}_{%
\widehat{\mu}} \rightarrow M$, and accordingly $L_{\Delta , \xi}$ has the
multipliers 
\begin{equation}  \label{multLminusP}
\begin{cases}
e_{\lambda_{\alpha}}(v)=e^{i \xi_{\alpha}} \\ 
e_{\lambda_{n+\alpha}}(v)= e^{i \xi_{n+\alpha}+2 \pi i \frac{\delta_{\alpha}%
}{\delta_{n}}\widehat{\mu}_{\alpha}},%
\end{cases}
\alpha =1,...,n. \hspace*{180pt}
\end{equation}
By $(\ref{xitomu})$ expressing $\widehat{\mu}$ in terms of $\xi$, we compute
the RHS of $(\ref{multLminusP})$ and rewrite the result as 
\begin{equation}  \label{multL}
\begin{cases}
e_{\lambda_{\alpha}}(v)=e^{i \xi_{\alpha}} \\ 
e_{\lambda_{n+\alpha}}(v)=e^{i\overset{n}{\underset{\beta =1}{\sum}} \frac{%
\tau_{\alpha \beta}}{\delta_{\beta}}\xi_{\beta}},%
\end{cases}
\alpha =1,...,n. \hspace*{230pt}
\end{equation}

\noindent By Lemma \ref{L=P+L} below $L_{\Delta, \xi}$ is holomorphically
trivial so that $\mathcal{L}_{[\xi]}\cong \mathrm{{P}_{\widehat{\mu}}}$,
proving $i)$.\newline

\noindent \hspace*{12pt} For $ii)$ of the claim above, the action $2 \pi \nu
\circ (x, \xi, \sigma) = (x, \xi + 2 \pi \nu, e^{- 2 \pi i\, <\nu, x>}\,
\sigma)$ gives $2 \pi \nu \circ (0, \xi , \sigma) = (0, \xi+2 \pi \nu,
\sigma)$ at $x=0$. Since $\sigma$ is unchanged, it follows that $\mathscr{P}%
|_{\{0\}\times M^{*}}$ has trivial multipliers and therefore is a
holomorphically trivial line bundle on $\{0\}\times M^{*}$, proving $ii)$.
\end{proof}

\begin{lemma}
\label{L=P+L} For any given $\xi \in V^{\ast}$ the above line bundle $%
L_{\Delta, \xi}$ with multipliers (\ref{multL}) is holomorphically trivial.
\end{lemma}

\begin{proof}
We claim that 
\begin{equation*}
\Phi _{\xi }(v)=e^{i\,\overset{n}{\underset{\alpha =1}{\sum }}a_{\alpha
}z_{\alpha }},\hspace*{10pt}a_{\alpha }=\frac{\xi _{\alpha }}{\delta
_{\alpha }}\hspace*{160pt}
\end{equation*}%
is a global section of $L_{\Delta ,\xi }\rightarrow M$. From the definition
of $\Phi _{\xi }(v)$, one sees that 
\begin{equation}
\begin{cases}
\Phi _{\xi }(z_{1},...,z_{\alpha }+\delta _{\alpha },...z_{n})=e^{i\,\xi
_{\alpha }}\Phi _{\xi }(z_{1},...,z_{n}) \\ 
\Phi _{\xi }(z_{1}+\tau _{\alpha 1},z_{2}+\tau _{\alpha 2},...,z_{n}+\tau
_{\alpha n})=e^{i\overset{n}{\underset{\beta =1}{\sum }}\frac{\tau _{\alpha
\beta }}{\delta _{\beta }}\xi _{\beta }}\Phi _{\xi }(z_{1},...,z_{n}),%
\end{cases}%
\alpha =1,...,n.\hspace*{60pt}
\end{equation}%
By $(\ref{multL})$ this proves the claim, and hence the lemma since $\Phi
_{\xi }(v)$ is nowhere vanishing.
\end{proof}

\noindent \hspace*{12pt} Now we are in a position to compute the curvature $%
\Theta (E,h_{E})$ (see Corollary $\ref{ThetaE}$ below). By proposition $\ref%
{P to P}$ that $P\cong \widetilde{Iso}^{\ast }\mathscr{P}$, we can pull back
the metric $h_{\mathscr{P}}$ and the connection $\nabla ^{\mathscr{P}%
}=d+i\xi $ to $P$ and write 
\begin{equation}
h_{P}:=\widetilde{Iso}^{\ast }h_{\mathscr{P}},\hspace*{10pt}\nabla ^{P}:=%
\widetilde{Iso}^{\ast }\nabla ^{\mathscr{P}}\mbox{\hspace*{10pt} on\ $P$}.
\label{hPCP}
\end{equation}%
Write $\Theta _{P}$ for the curvature of $\nabla ^{P}$. Combining $(\ref{dA}%
) $ and Proposition $\ref{Prop 3.4}$ using $(\ref{complex linear})$, one
knows that $(Id\times \varphi _{L_{0}})^{\ast }h_{P}$ and $h_{(Id\times
\varphi _{L_{0}})^{\ast }P}$ (see (\ref{hP})) differ at most by a
multiplicative constant $c$ (\cite[Prop. 8.1]{CCT}). Comparing them on $%
\{0\}\times M$, one sees that $c=1$ and hence that they are identical on $%
M\times M$. More precisely we have

\begin{proposition}
\label{Prop5.4} Recalling $h_{(Id \times \varphi_{L_{0}})^{*}P}$ and $%
\Theta_{(Id \times \varphi_{L_{0}})^{*}P}$ on $(Id \times
\varphi_{L_{0}})^{*}P \rightarrow M \times M $ (see (\ref{hP}) and (\ref%
{CuvatureP})), one has $(Id \times \varphi_{L_{0}})^{*}
\Theta_{P}=\Theta_{(Id \times \varphi_{L_{0}})^{*}P} $ and $(Id \times
\varphi_{L_{0}})^{*} h_{P}=h_{(Id \times \varphi_{L_{0}})^{*}P} $ on $M
\times M$.
\end{proposition}

\noindent \hspace*{12pt} To proceed further, we consider the following
bundles: \noindent

\begin{definition}
\label{Def5.6}

$i)$ Define the line bundles 
\begin{equation}
\widetilde{E}:=\pi _{1}^{\ast }L_{0}\otimes P\rightarrow M\times \widehat{M},%
\hspace*{10pt}\widetilde{E}^{\prime }:=\pi _{1}^{\ast }L_{0}\otimes
(Id\times \varphi _{L_{0}})^{\ast }P\rightarrow M_{1}\times M_{2}  \label{E2}
\end{equation}%
\noindent where $M_{1}\cong M_{2}\cong M$.\vspace*{-8pt}\newline

$ii)$ Define the direct images 
\begin{equation}
E:=(\pi _{2})_{\ast }\widetilde{E}\rightarrow \widehat{M};\hspace*{10pt}%
E^{\prime }:=(\pi _{2}^{\prime })_{\ast }\widetilde{E}^{\prime }\rightarrow
M_{2}  \label{E'}
\end{equation}%
\noindent where $\pi _{2}:M\times \widehat{M}\rightarrow \widehat{M}$, $\pi
_{2}^{\prime }:M_{1}\times M_{2}\rightarrow M_{2}$ are projections. 
\end{definition}

Note that $E$ and $E^{\prime }$ are vector bundles by Grauert's direct image
theorem (cf. \cite[Cor. 12.9]{Harts} and \cite[Section 3 of Chapter 3]%
{Banica}) since the dimension functions $h^{0}(M,L_{\widehat{\mu }})$ and $%
h^{0}(M,L_{\mu })$ are constant in $\widehat{\mu }$ and $\mu $ (cf. Lemmas %
\ref{Lemma 2.1} and \ref{Lemma 2.2}) respectively. \newline

Using the identification $(P,h_{P},\nabla _{P})\cong (\widetilde{Iso}^{\ast }%
\mathscr{P},\ \widetilde{Iso}^{\ast }h_{\mathscr{P}},\ \widetilde{Iso}^{\ast
}\nabla ^{\mathscr{P}})$ we equip $\widetilde{E}$ and $\widetilde{E^{\prime }%
}$ with the metric 
\begin{equation}
h_{\widetilde{E}}=\pi _{1}^{\ast }h_{L_{0}}\otimes h_{P}%
\mbox{\hspace{10pt} and
\hspace*{10pt}}h_{\widetilde{E}^{\prime }}=\pi _{1}^{\ast }h_{L_{0}}\otimes
h_{(Id\times \varphi _{L_{0}})^{\ast }P}
\end{equation}%
respectively (cf. Lemma \ref{Lemma 2.1} and (\ref{hP})). One has $h_{%
\widetilde{E}^{\prime }}=(Id\times \varphi _{L_{0}})^{\ast }h_{\widetilde{E}%
} $ by Proposition \ref{Prop5.4}. We can now equip the vector bundle $E$
with the metric given by the $L^{2}$ inner product using $h_{\widetilde{E}}$
on the fibers $E_{\widehat{\mu }}=H^{0}(M,L_{\widehat{\mu }})$, and
similarly $E^{\prime }$ with the metric given by that using $h_{\widetilde{E}%
^{\prime }} $ on $E^{\prime }{}_{\mu }=H^{0}(M,L_{\mu })$.

\begin{notation}
\label{N-5-6a} We denote by $h_{E}$ and $h_{E^{\prime }}$ those metrics just
obtained on $E$ and $E^{\prime },$ respectively.
\end{notation}

The statement and proof of the next proposition are the same as those of
Proposition 8.3 of \cite{CCT}:

\begin{proposition}
\label{Prop5.6} With the notation above, we have $(\widetilde{E}^{\prime
},h_{\widetilde{E}^{\prime }})=(Id\times \varphi _{L_{0}})^{\ast }(%
\widetilde{E},h_{\widetilde{E}})$. As a consequence, $(E^{\prime
},h_{E^{\prime }})=\varphi _{L_{0}}^{\ast }(E,h_{E})$ with the curvature $%
\Theta (E^{\prime },h_{E^{\prime }})=\varphi _{L_{0}}^{\ast }\Theta
(E,h_{E}) $.
\end{proposition}

\noindent \hspace*{12pt} Recall that $K\rightarrow M_{2}$ has the fiber $%
K_{\mu }=H^{0}(M,\widetilde{K}|_{M\times \{\mu \}})$ (cf. Section 4). By $(%
\ref{Kdirect})$, $(\ref{E2})$ and $(\ref{E'})$ we have $K=E^{\prime }\otimes
L_{0}$, that is $E^{\prime }=K\otimes L_{0}^{\ast }$ (where $L_{0}^{\ast }$
is the dual bundle of $L_{0}$), and it follows from Theorem \ref{theo4.3}
that 
\begin{equation}
E^{\prime }(=K\otimes L_{0}^{\ast })=\underset{m\in \mathfrak{M}}{\oplus }%
(K_{m}\otimes L_{0}^{\ast })\cong \underset{m\in \mathfrak{M}}{\oplus }%
L_{0}^{\ast }.  \label{E'KL0}
\end{equation}

\bigskip

\noindent \hspace*{12pt} In view of Theorem \ref{theo4.3}, (\ref{E'KL0}) and 
$\Theta _{L_{0}}=-\partial \overline{\partial }\log h_{L_{0}}=\pi \underset{%
\alpha ,\beta =1}{\sum }W_{\alpha \beta }\,d\mu _{\alpha }\wedge d\overline{%
\mu }_{\beta }$ (Lemma \ref{Lemma 2.1}), the curvature of $E^{\prime }$ is
now immediate:

\begin{theorem}
\label{theo5.7} Let us denote by $I_{\Delta \times \Delta }$ the $\Delta
\times \Delta$ identity matrix where $\Delta = \det (\Delta_{\delta})=%
\overset{n}{\underset{\alpha=1}{\prod}} \delta_{\alpha}=\dim H^{0}(M,L_{0}).$
Then 
\begin{equation}  \label{curvE'}
\Theta (E^{\prime },h_{E^{\prime }})=-\Theta _{L_{0}}\cdot I_{\Delta \times
\Delta }=\big(-\pi \overset{n}{\underset{\alpha ,\beta =1}{\sum }} W_{\alpha
\beta }\,d\mu _{\alpha }\wedge d\overline{\mu }_{\beta } \big) \cdot
I_{\Delta \times \Delta }\vspace*{-16pt}
\end{equation}
\begin{equation}  \label{c1E'}
c_{1}(E^{\prime },h_{E^{\prime }})=-\frac{i}{2\pi}\Delta \Theta_{L_{0}}= - 
\frac{i}{2} \Delta \overset{n}{\underset{\alpha ,\beta =1}{\sum }}W_{\alpha
\beta }\,d\mu _{\alpha }\wedge d\overline{\mu }_{\beta },\hspace*{10pt} %
\mbox{\ on\ }M.
\end{equation}
\end{theorem}

\begin{remark}
\label{Rmk5.9} By (\ref{complaxreal}), we have 
\begin{equation}
d\mu _{\alpha }=\delta _{\alpha }dx_{\alpha }+\overset{n}{\underset{k=1}{%
\sum }\,}\tau _{\alpha k}\text{ }dx_{n+k}\text{, \ \ \ }d\overline{\mu }%
_{\alpha }=\delta _{\alpha }dx_{\alpha }+\overset{n}{\underset{k=1}{\sum }\,}%
\overline{\tau }_{\alpha k}\text{ }dx_{n+k}\text{, \ \ }\alpha =1,...,n%
\hspace*{40pt}
\end{equation}%
so that $c_{1}(E^{\prime },h_{E^{\prime }})=-\frac{i}{2}\Delta \overset{n}{%
\underset{\alpha ,\beta =1}{\sum }}W_{\alpha \beta }\,d\mu _{\alpha }\wedge d%
\overline{\mu }_{\beta }=-\Delta \overset{n}{\underset{i=1}{\sum }}\,\delta
_{i}dx_{i}\,\wedge dx_{n+i}=-\Delta \cdot c_{1}(L_{0})$.
\end{remark}

\begin{corollary}
\label{ThetaE} On $\widehat{M}$, we have the curvature $\Theta (E,h_{E})=%
\big( -\pi \overset{n}{\underset{\alpha ,\beta =1}{\sum }}W_{\alpha \beta
}\, \frac{\delta_{\alpha}\delta_{\beta}}{\delta_{n} \delta_{n}}\, d\widehat{%
\mu }_{\alpha }\wedge d\overline{\widehat{\mu }}_{\beta }\, \big) \cdot
I_{\Delta\times \Delta }$ and the first Chern class $c_{1}(E,h_{E})=\big( -%
\frac{i}{2} \Delta \overset{n}{\underset{\alpha ,\beta =1}{\sum }}W_{\alpha
\beta }\, \frac{\delta_{\alpha}\delta_{\beta}}{\delta_{n} \delta_{n}} \,d%
\widehat{\mu}_{\alpha }\wedge d\overline{\widehat{\mu} }_{\beta }\, \big) $.
\end{corollary}

\vspace*{4pt}

\begin{proof}
With the notation from Definition $\ref{Def5.6}$, Proposition \ref{Prop5.6}
gives us the following commutative diagram 
\begin{equation*}
\begin{tikzcd}[column sep=2cm] \widetilde{E} \arrow{r}{ (Id \times
\varphi_{L_{0}})^{*}} \arrow[swap]{d}{\pi_{2}} & \widetilde{E}^{\, \prime}
\arrow{d}{\pi_{2}^{\prime}} \\E \arrow{r}{\varphi_{L_{0}}^{*}}& E^{\prime}
\end{tikzcd}\hspace*{160pt}
\end{equation*}
Since $\varphi_{L_{0}} : M \rightarrow \widehat{M}$ is a local
diffeomorphism and we have the formula for $\Theta(E^{\prime }, h_{E^{\prime
}})$ and $c_{1}(E^{\prime }, h_{E^{\prime }})$ of $E^{\prime }\rightarrow M$
(see (\ref{curvE'}), (\ref{c1E'})), the corollary follows from the equations 
$\widehat{\mu}_{\alpha}=\frac{\delta_{n}}{\delta_{\alpha}} \mu_{\alpha}$
(see Property $\ref{Prop 2.3})$.
\end{proof}

\vspace*{8pt} \noindent \hspace*{12pt} On $M^{\ast }=V^{\ast }/2\pi \Lambda
^{\ast }$ (cf. Lemma \ref{Iso}), if we write $V^{\ast }:=Hom_{\mathbb{R}}(V,%
\mathbb{R})=\{\xi _{1}dx_{1}+...+\xi _{2n}dx_{2n}\mid \xi _{i}\in \mathbb{R}%
\}=\{\eta _{1}\left( 2\pi dx_{1}\right) +...+\eta _{2n}\left( 2\pi
dx_{2n}\right) \mid \eta _{i}\in \mathbb{R}\}$ and $2\pi \Lambda ^{\ast }=%
\mathbb{Z}\{2\pi dx_{1},...,2\pi dx_{2n}\}$, then $\eta _{1},...,\eta _{2n}$
are the dual real coordinates on $V^{\ast }$ with $d\eta _{1},...,d\eta _{n}$
the corresponding 1-forms on $M^{\ast }$. It is known that $H^{k}(M^{\ast }, 
\mathbb{Z})=\mathbb{Z}\{d\eta_{I}\}_{\#I=k}$ (\cite[p.302]{GH}).

\begin{theorem}
\label{5.10} By the map $Iso : \widehat{M} \rightarrow M^{*}$ (Lemma \ref%
{Iso}) we have, with $\Delta=\overset{n}{\underset{\alpha=1}{\prod}}%
\delta_{\alpha}$ 
\begin{equation}  \label{5.19}
c_{1}((Iso^{-1})^{*}E, (Iso^{-1})^{*}(h_{E}))=- \overset{n}{\underset{i=1}{%
\sum}} \frac{\Delta}{\delta_{i}}\, d \eta_{i} \wedge d\eta_{n+i} \in
H^{2}(M^{*}, \mathbb{Z}).\hspace{60pt}
\end{equation}
\end{theorem}

\begin{remark}
\label{R-5-12} $\omega ^{\vee }\coloneqq { \overset{n}{\underset{i=1}{\sum}}}%
\,{\ \frac{\Delta }{\delta _{i}}}\ d\eta _{i}\wedge d\eta _{n+i}$ can be
regarded as a \textquotedblleft dual polarization" on the dual Abelian
variety $\widehat{M}$ (via the map $Iso$) while $\omega =\overset{n}{%
\underset{i=1}{\sum }}\delta _{i}dx_{i}\wedge dx_{n+i}$ in (\ref{2.1}) is
the polarization on $M$. For the elliptic curve $\omega ^{\vee }=d\eta
_{1}\wedge d\eta _{2}$ represents 1 in $H^{2}(\widehat{M},\mathbb{Z})=%
\mathbb{Z}$.
\end{remark}

\begin{proof}[Proof of Theorem \protect\ref{5.10}]
\ From $(\ref{mutoxi})$, $\widehat{\mu }_{\alpha }=\frac{\delta _{n}}{\delta
_{\alpha }}\mu _{\alpha },\alpha =1,...,n$ (Property \ref{Prop 2.3}), and $%
2\pi \eta _{i}=\xi _{i},i=1,...,2n$, one sees that\vspace*{4pt} 
\begin{equation}
\begin{pmatrix}
d\eta _{1} \\ 
\vdots \\ 
d\eta _{n}%
\end{pmatrix}%
=\Delta _{\delta }\,W%
\begin{pmatrix}
d\mu _{12} \\ 
\vdots \\ 
d\mu _{n2}%
\end{pmatrix}%
,\hspace*{10pt}%
\begin{pmatrix}
d\eta _{n+1} \\ 
\vdots \\ 
d\eta _{2n}%
\end{pmatrix}%
=-%
\begin{pmatrix}
d\mu _{11} \\ 
\vdots \\ 
d\mu _{n1}%
\end{pmatrix}%
+(\func{Re}Z)\,W%
\begin{pmatrix}
d\mu _{12} \\ 
\vdots \\ 
d\mu _{n2}%
\end{pmatrix}%
\hspace*{30pt}\vspace*{8pt}  \label{mutoeta}
\end{equation}

\noindent where $(\Delta _{\delta },Z)$ is the period matrix of $\Lambda $
(see (\ref{period})). We want to calculate the 2-form $-\,\omega ^{\vee }=-%
\overset{n}{\underset{\alpha =1}{\sum }}\frac{\Delta }{\delta _{i}}d\eta
_{i}\wedge d\eta _{n+i}$ and relate it to $c_{1}(E^{\prime },h_{E^{\prime
}})=-\frac{i}{2}\Delta \overset{n}{\underset{\alpha ,\beta =1}{\sum }}%
W_{\alpha \beta }\,d\mu _{\alpha }\wedge d\overline{\mu }_{\beta }$. First,
writing $f_{\alpha }=\overset{n}{\underset{\beta =1}{\sum }}\,W_{\alpha
\beta }\,d\mu _{\beta 2}$ ($\alpha =1,...,n$) we see from (\ref{mutoeta})
that 
\begin{align}
& \left( (\Delta _{\delta }^{-1}\Delta )\,%
\begin{pmatrix}
d\eta _{1} \\ 
\vdots \\ 
d\eta _{n}%
\end{pmatrix}%
\right) ^{T}\wedge (\func{Re}Z)W%
\begin{pmatrix}
d\mu _{12} \\ 
\vdots \\ 
d\mu _{n2}%
\end{pmatrix}%
\hspace*{210pt}  \label{DeltaT} \\
& =\left( (\Delta _{\delta }^{-1}\Delta )\,\Delta _{\delta }W%
\begin{pmatrix}
d\mu _{12} \\ 
\vdots \\ 
d\mu _{n2}%
\end{pmatrix}%
\right) ^{T}\wedge (\func{Re}Z)W%
\begin{pmatrix}
d\mu _{12} \\ 
\vdots \\ 
d\mu _{n2}%
\end{pmatrix}%
\hspace*{70pt}  \notag \\
& =\Delta 
\begin{pmatrix}
f_{1} \\ 
\vdots \\ 
f_{n}%
\end{pmatrix}%
^{T}\wedge (\func{Re}Z)%
\begin{pmatrix}
f_{1} \\ 
\vdots \\ 
f_{n}%
\end{pmatrix}%
=\Delta \overset{n}{\underset{\alpha ,\beta =1}{\sum }}(\func{Re}Z)_{\alpha
\beta }\,f_{\alpha }\wedge f_{\beta }=0\hspace*{34pt}  \notag
\end{align}

\vspace*{8pt} \noindent since $\func{Re}Z$ is symmetric and ${f_{\alpha }}%
^{,}s$ are 1-forms. Therefore, by (\ref{mutoeta}) again and using (\ref%
{DeltaT}) 
\begin{align}
\hspace*{40pt}\varphi _{L_{0}}^{\ast }Iso^{\ast }(-\omega ^{\vee })&
=\varphi _{L_{0}}^{\ast }Iso^{\ast }(-\overset{n}{\underset{\alpha =1}{\sum }%
}\frac{\Delta }{\delta _{i}}\,d\eta _{i}\wedge d\eta _{n+i})=\left( (\Delta
_{\delta }^{-1}\Delta )\,\Delta _{\delta }W%
\begin{pmatrix}
d\mu _{12} \\ 
\vdots \\ 
d\mu _{n2}%
\end{pmatrix}%
\right) ^{T}\wedge 
\begin{pmatrix}
-d\mu _{11} \\ 
\vdots \\ 
-d\mu _{n1}%
\end{pmatrix}
\notag \\
& =-\Delta \overset{n}{\underset{\alpha ,\beta =1}{\sum }}W_{\alpha \beta
}\,d\mu _{\alpha 1}\wedge d\mu _{\beta 2}=-\frac{i}{2}\Delta \overset{n}{%
\underset{\alpha ,\beta =1}{\sum }}W_{\alpha \beta }\,d\mu _{\alpha }\wedge d%
\overline{\mu }_{\beta }\overset{(\ref{c1E'})}{=}c_{1}(E^{\prime
},h_{E^{\prime }}).
\end{align}%
Now (\ref{5.19}) follows since $\varphi _{L_{0}}^{\ast
}(c_{1}(E,h_{E}))=c_{1}(E^{\prime },h_{E^{\prime }})$ by Proposition \ref%
{Prop5.6} and $\varphi _{L_{0}}$ is a local diffeomorphism.
\end{proof}

\begin{remark}
The motivation for the \textit{a priori} choice $\omega ^{\vee }={\ \overset{%
n}{\underset{i=1}{\sum }}}\,{\ \frac{\Delta }{\delta _{i}}}\ d\eta
_{i}\wedge d\eta _{n+i}$ in Theorem \ref{5.10} is that suppose $%
c_{1}(E,h_{E})=\overset{n}{\underset{i=1}{\sum }}\,(-a_{i})\,d\eta
_{i}\wedge d\eta _{n+i}\in H^{2}(\widehat{M},\mathbb{Z})$ for some $a_{i}\in 
\mathbb{Z}$ (as yet unknown). Then ($\widehat{M}$ and $M^{\ast }$ are
identified via $Iso$) 
\begin{equation}
\int_{M}(\varphi _{L_{0}}^{\ast }c_{1}(E,h_{E}))^{n}=\deg (\varphi
_{L_{0}})\int_{\widehat{M}}c_{1}^{n}(E,h_{E})=(-1)^{\frac{n(n+1)}{2}%
}(n!)\Delta ^{2}\,\overset{n}{\underset{i=1}{\prod }}a_{i}  \label{up}
\end{equation}%
where $\deg (\varphi _{L_{0}})=\Delta ^{2}$ by \cite[pp. 315-317]{GH}.
Furthermore, by Remark \ref{Rmk5.9} 
\begin{equation}
\int_{M}c_{1}^{n}(E^{\prime },h_{E^{\prime }})=\int_{M}(-\Delta \overset{n}{%
\underset{i=1}{\sum }}\,\delta _{i}dx_{i}\,\wedge dx_{n+i})^{n}=(-1)^{\frac{%
n(n+1)}{2}}(n!)\Delta ^{n+1}.  \label{down}
\end{equation}

\noindent We see from (\ref{up}) and (\ref{down}) that for $\varphi
_{L_{0}}^{\ast }c_{1}(E,h_{E})=c_{1}(E^{\prime },h_{E^{\prime }})$ to hold,
one may try $a_{i}=\frac{\Delta }{\delta _{i}}$ because $\Delta =\overset{n}{%
\underset{\alpha =1}{\prod }}\delta _{\alpha }$.\vspace*{16pt}
\end{remark}


\begin{thebibliography}{99}
\bibitem{ACGH} E. Arbarello, M. Cornalba, P. A. Griffiths, J. Harris, 
\textit{Geometry of algebraic curves, vol. I, }Grundlehren der
mathematischen Wissenschaften (vol. 267) 1985.

\bibitem{Banica} C. Banica and O. Stanasila, \textit{Algebraic Methods in
the global theory of complex spaces}, Editura Academiei, John Wiley \& Sons,
1976.

\bibitem{Bern} B. Berndtsson, \textit{Curvature of vector bundles associated
to holomorphic fibrations}, Annals of Math. \textbf{169} (2009) 531-560.

\bibitem{CCT} C.-H. Chang, J.-H. Cheng and I-H. Tsai, \textit{Theta
Functions and Adiabatic Curvature on an Elliptic Curve}, Journal of
Geometric Analysis 32: 84 (2022) No. 3.

\bibitem{CCT2} C.-H. Chang, J.-H. Cheng and I-H. Tsai, \textit{Correction to}
\textit{Theta Functions and Adiabatic Curvature on an Elliptic Curve},
Journal of Geometric Analysis 32:208 (2022) No. 7.

\bibitem{CCT3} C.-H. Chang, J.-H. Cheng and I-H. Tsai,\textit{\ Spectral
bundles on an Abelian variety}, draft.

\bibitem{DK} S. K. Donaldson, P. B. Kronheimer, \textit{The Geometry of
Four-Manifolds}, Oxford Science Publication, 1997.

\bibitem{GH} P. Griffiths, J. Harris, \textit{Principles of Algebraic
Geometry}, Wiley Classic Library, 1984.

\bibitem{Harts} R. Harshorne, \textit{Algebraic Geometry}, Graduate Text in
Mathematics \textbf{52}, Springer-Verlag, New York 1977

\bibitem{Kempf1} G, Kempf, \textit{Toward the inversion of abelian integrals}%
, II, Amer. J. Math. \textbf{101} (1979) 184-202.

\bibitem{Kempf2} G. Kempf, \textit{A problem of Narasimhan}, Curves,
Jacobians, and abelian varieties (Amherst, MA, 1990), 283-286, Contemp.
Math. \textbf{136} Amer. Math. Soc., Providence, RI, 1992.

\bibitem{M} D. Mumford, \textit{Tata Lectures on Theta I}, Birkh\"{a}user
Boston 1983.

\bibitem{Pol} A. Polishchuk, \textit{Abelian varieties, theta functions and
the Fourier transform}, Cambridge University Press 2003.

\bibitem{P1} C. T. Prieto, \textit{Holomorphic\ spectral\ geometry\ of\
magnetic\ $Schr\ddot{o}dinger$\ operators\ on\ Riemann\ surfaces},
Differential Geometry and its Applications \textbf{24} (2006) 288-310.

\bibitem{P2} C. T. Prieto, \textit{Fourier-Mukai transform and adiabatic
curvature of spectral bundles for Landau Hamiltonians on Riemann surfaces},
Commun. Math. Phys. \textbf{265} (2006) 373-396.

\bibitem{Roq} P. Roquette, \textit{Analytic theory of elliptic functions
over local fields}, Vandenhoeck \& Ruprecht in G\"{o}ttingen 1970.

\bibitem{TW2} W.-K. To and L. Weng, \textit{Curvature of the $L^{2}$-metric on
the direct image of a family of Hermitian-Einstein vector bundles}, Amer. J.
Math. \textbf{120} (1998) 649-661.

\bibitem{To} W.-K. To and L. Weng, \textit{$L^{2}$-metrics, projective
flatness and families of polarized Abelian varieties}, Trans, Amer. Math
Soc. \textbf{356} (2004) No. 7, 2685-2707










\end{thebibliography}
\end{document}